\documentclass[10pt]{article}
\DeclareMathSizes{10}{9}{7}{5}
\usepackage{amsmath,amssymb,amsthm,amsfonts}
\usepackage{bm}
\usepackage{mathrsfs}
\usepackage{graphicx}
\usepackage{mathtools}
\usepackage{braket}
\usepackage{comment}
\usepackage{enumitem}
\setlist[enumerate,1]{label=(\roman*),font=\upshape}
\usepackage{indentfirst}
\usepackage{booktabs}
\usepackage{placeins}
\usepackage{xurl}
\usepackage{authblk}
\usepackage[hidelinks]{hyperref}
\usepackage{bookmark}
\usepackage{cleveref}
\usepackage{xcolor}
\mathtoolsset{showonlyrefs}

\theoremstyle{plain}
\newtheorem{thm}{Theorem}[section]
\newtheorem{pro}{Proposition}[section]
\newtheorem{lem}{Lemma}[section]
\newtheorem{cor}{Corollary}[section]

\theoremstyle{definition}

\newtheorem{Khyp}{Assumption}

\newtheorem{Vhyp}{Assumption}

\crefname{Khyp}{Assumption}{Assumptions}
\crefname{Vhyp}{Assumption}{Assumptions}
\crefname{Ghyp}{Assumption}{Assumptions}
\newtheorem{rem}{Remark}[section]
\numberwithin{equation}{section}

\title{\Large Marginal Composite Likelihood Estimation of Bergomi Models from High-Frequency Option Prices}

\author{Haruki Tomita\thanks{Supported by JST SPRING, Grant Number JPMJSP2138.}}

\affil{\small Graduate School of Engineering Science, The University of Osaka}
\date{}

\begin{document}
\maketitle

\section{Introduction}
Rough volatility models have become an important framework for describing
the irregular dynamics of financial volatility. Empirical evidence of
volatility fluctuations with low local regularity motivated specifications
with a Hurst parameter below one half; see
\cite{GatheralJaissonRosenbaum2018}. The relevance of roughness to option
markets also has a theoretical basis through its connection with the
short-maturity behavior of the implied-volatility skew, as studied in
\cite{Fukasawa2021}. Within the forward variance framework
\cite{Bergomi2016}, the rough Bergomi model \cite{Bayer02062016}
represents this irregularity through a singular power-law kernel.
Motivated by this model, we develop a parametric inference theory for
Bergomi-type models with a general class of positive kernels that may
be singular at the origin.

A substantial literature studies rough volatility using return data
and realized measures of volatility; see, for example,
\cite{FukasawaTakabatakeWestphal2022}.
Models that separate short-term roughness from long-term persistence,
together with noise-robust estimation methods, are developed in
\cite{BennedsenLundePakkanen2022}.
A generalized method of moments approach to fractional log-normal
stochastic volatility models, accounting for measurement error in
realized variance, is proposed in \cite{BolkoEtAl2023}.
Research using option observations to study rough volatility remains
comparatively limited, although important contributions already exist.
Evidence of roughness based on short-maturity implied volatility is
provided in \cite{LivieriEtAl2018}, while a nonparametric estimator of
volatility roughness based on short-dated option portfolios is
constructed in \cite{10.3150/24-BEJ1762}.
These studies form part of a broader literature on option-based
inference, including the parametric option-panel methodology of
\cite{AndersenFusariTodorov2015} and the nonparametric spot-volatility
estimation of \cite{Todorov2019}.
Our focus is the joint estimation of parameters governing a positive,
potentially singular forward variance kernel from high-frequency
observations across time and maturity.

We work under a martingale measure equivalent to the physical measure,
with model filtration $(\mathscr F_t)_{0\leq t\leq1}$.
All expectations below are taken under this martingale measure.
Let $W$ be a one-dimensional $(\mathscr F_t)$-Brownian motion.
The spot variance is defined by
\begin{equation}
    V_t
    =
    V_0^t
    \exp\left\{
        \int_0^t k_\theta(t-s)dW_s
        -
        \frac12\int_0^t k_\theta(t-s)^2ds
    \right\},
\end{equation}
where $V_0^\cdot$ is a positive deterministic initial curve and
$\theta$ is an unknown finite-dimensional parameter.
The positive kernel $k_\theta$ is allowed to be singular at the
origin, subject to the assumptions stated below, which ensure
square integrability near zero.
The normalization gives $E[V_t]=V_0^t$.
Defining the forward variance by
$V_t^u:=E[V_u\mid\mathscr F_t]$, for $t\leq u$,
we obtain
\begin{equation}
    \frac{dV_t^u}{V_t^u}
    =
    k_\theta(u-t)dW_t,
    \qquad t<u.
\end{equation}
Thus, the forward variance curve is an infinite-dimensional state
driven by a single Brownian motion.
Our principal example is the pure power-law specification
\begin{equation}
    k_\theta(x)=\eta x^{H-1/2},
    \qquad
    \theta=(\eta,H),\quad
    \eta>0,\quad 0<H\leq1/2.
\end{equation}
For $H<1/2$, the variance dynamics above correspond to those
of the rough Bergomi model.
The parameter $H$ controls the singularity at the origin and the
local regularity of volatility, while the kernel as a whole
determines how a common shock propagates across maturities.
Our theory covers a general class of positive kernels,
including both singular and regular specifications.

We assume that the underlying asset price $S_t$ is a strictly positive
continuous martingale with spot variance $V_t$.
Asset-price jumps are excluded, and interest rates and dividends are
set to zero.
Since the forward variance curve is not directly observed, we introduce
the cumulative forward variance.
Under standard integrability conditions, it admits the option-price
representation \cite{BrittenJonesNeuberger2000}
\begin{equation}
\begin{aligned}
    I_t^T
    &:=
    \int_t^T V_t^u du \\
    &=
    2\int_0^{S_t}
    \frac{E[(K-S_T)_+\mid\mathscr F_t]}{K^2}dK
    +
    2\int_{S_t}^{\infty}
    \frac{E[(S_T-K)_+\mid\mathscr F_t]}{K^2}dK,
    \qquad t<T.
\end{aligned}
\end{equation}
This representation motivates using cumulative forward variance
observations obtained from option-price time series.

Under the Bergomi model, the cumulative forward variance satisfies
\begin{equation}
    I_t^T
    =
    I_0^T
    -
    \int_0^t V_s ds
    +
    \int_0^t
    \left(
        \int_s^T
        k_\theta(u-s)
        \frac{\partial I_s^u}{\partial u}du
    \right)dW_s,
    \qquad t<T,
\end{equation}
where $\partial I_s^u/\partial u=V_s^u$.
This formulation, developed in our earlier work
\cite{fukasawa2026bergomi}, turns kernel estimation into a
diffusion-coefficient estimation problem for an infinite-dimensional
stochastic differential equation.

Our statistical experiment consists of exact observations of the
cumulative forward variances $\{I_{t_i}^{T_j}\}$, with no observation
or reconstruction error, on a deterministic time--maturity grid
within $[0,1]$:
\begin{equation}
    t_i=i\Delta_n,\qquad
    \Delta_n=1/n,\qquad
    0=T_0<T_1<\cdots<T_d\leq1.
\end{equation}
Writing
$\delta_j:=T_j-T_{j-1}$ and
$\Delta_d:=\max_{1\leq j\leq d}\delta_j$,
we consider the joint limit
\begin{equation}
    n\longrightarrow\infty,\qquad
    \Delta_d\longrightarrow0,\qquad
    T_d\longrightarrow1.
\end{equation}
This setting presents two difficulties.
First, the infinite-dimensional state is driven by highly degenerate
noise: the leading covariance of a vector of maturity increments has
rank one.
Second, the kernel may be singular at the maturity diagonal, as in
the pure power-law specification with $H<1/2$.

Parametric inference for parabolic SPDEs from high-frequency
spatio-temporal observations is studied in
\cite{KainoUchida2021,TonakiKainoUchida2023}.
These models have infinitely many independent Brownian noise
coordinates, whereas the maturity curve in our setting is driven
by a single Brownian motion.
Related inference from integrated curve observations is developed in
\cite{Schroers2026}, where discrete observations of log bond prices
are used to estimate the quadratic covariation of the latent
semimartingale driver in term structure models.
These log prices are negative maturity integrals of the underlying
forward rate curves.
The estimation target in the present paper is the finite-dimensional
parameter governing the forward variance kernel.

In \cite{fukasawa2026bergomi}, we constructed an estimator for
essentially the same observation problem.
Following the discretization strategy of
\cite{genon1993estimation}, we used an Euler--Maruyama approximation
to the cumulative forward variance dynamics and replaced the
resulting degenerate Gaussian approximation by a nondegenerate proxy.
That construction establishes consistency and asymptotic mixed
normality for regular kernels, but it leaves three limitations.
Its smoothness assumptions exclude the singular power-law kernel
of the rough Bergomi model.
Its identification and local nondegeneracy assumptions require
constant-sign conditions on kernel differences and directional
derivatives over sufficiently short intervals adjacent to the origin.
Moreover, the earlier contrast involves an auxiliary parameter
$\epsilon>0$, introduced for technical purposes in the proofs to keep
the logarithm argument and the denominator bounded away from zero.
The resulting estimator depends on this parameter, and its choice
can substantially affect the estimates, as illustrated in that study.
The present paper addresses these limitations by changing the
construction of the estimating function from the outset.

Our first step is to take maturity differences before forming time
increments.
For $t<T_{j-1}$, define
\begin{equation}
    J_t^j
    :=
    I_t^{T_j}-I_t^{T_{j-1}}
    =
    \int_{T_{j-1}}^{T_j}V_t^u du.
\end{equation}
The maturity-independent drift in the equation for $I_t^T$ cancels,
and It\^o's formula gives
\begin{equation}
    d\log J_t^j
    =
    \sigma_t^j(\theta)dW_t
    -
    \frac12(\sigma_t^j(\theta))^2dt,
    \qquad
    \sigma_t^j(\theta)
    :=
    \frac{
        \int_{T_{j-1}}^{T_j}k_\theta(u-t)V_t^u du
    }{
        \int_{T_{j-1}}^{T_j}V_t^u du
    }.
\end{equation}
As the maturity mesh shrinks, we approximate the average of the
product by the product of the corresponding averages.
The forward variance factor then cancels with the denominator,
giving
\begin{equation}
    \sigma_t^j(\theta)\approx\bar k_\theta^j(t),
    \qquad
    \bar k_\theta^j(t)
    :=
    \frac1{\delta_j}
    \int_{T_{j-1}}^{T_j}k_\theta(u-t)du.
\end{equation}
This transformation removes the infinite-dimensional random state
from the leading diffusion coefficient.

Consequently, the logarithmic time increments
\begin{equation}
    X_{ij}
    :=
    \log J_{t_i}^j-\log J_{t_{i-1}}^j,
    \qquad
    j\in\Pi_i:=\{j\mid t_i<T_{j-1}\},
\end{equation}
admit the Gaussian approximation
\begin{equation}
    X_{ij}
    \approx
    \int_{t_{i-1}}^{t_i}\bar k_\theta^j(s)dW_s
    -
    \frac12q_{ij}(\theta),
    \qquad
    q_{ij}(\theta)
    :=
    \int_{t_{i-1}}^{t_i}\bar k_\theta^j(s)^2ds.
\end{equation}
The Gaussian proxy vectors are independent across disjoint time
intervals, although their maturity coordinates remain dependent.
Away from the maturity diagonal, their leading covariance after
normalization by $\sqrt{\Delta_n}$ is
\begin{equation}
    K_i(\theta)K_i(\theta)^\top,
    \qquad
    K_i(\theta)
    :=
    \bigl(\bar k_\theta^j(t_{i-1})\bigr)_{j\in\Pi_i}.
\end{equation}
Its rank-one structure makes the leading joint Gaussian likelihood
singular.
Because the components of $K_i(\theta)$ are positive, the marginal
variances determine the entire leading covariance matrix.
This observation motivates a contrast based on the marginal
Gaussian likelihoods.

After dropping parameter-independent terms, this construction gives
the weighted marginal composite likelihood contrast
\begin{equation}
    U_{n,d}(\theta)
    =
    \frac1n
    \sum_{i=1}^n
    \sum_{j\in\Pi_i}
    \delta_j
    \left[
        \log\frac{q_{ij}(\theta)}{\Delta_n}
        +
        \frac{X_{ij}^2}{q_{ij}(\theta)}
        +
        \frac{q_{ij}(\theta)}4
    \right].
\end{equation}
We define $\hat\theta_{n,d}$ as a minimizer over the parameter space.
The construction belongs to the composite likelihood framework
\cite{Lindsay1988,VarinReidFirth2011}.
The use of marginal likelihoods does not impose independence across
maturities: their dependence is retained in the limiting sandwich
covariance.
The proposed contrast contains no auxiliary parameter $\epsilon$.
Its evaluation requires neither reconstruction of the full forward
variance curve nor inversion of a maturity covariance matrix.

Our main results establish consistency and $\mathscr F_1$-stable
asymptotic normality for general positive kernel families, including
singular kernels, in the joint time--maturity limit.
Consistency requires no relative rate restriction between the two
meshes and allows the true parameter to lie on the boundary of the
parameter space.
For an interior true parameter, we establish $\mathscr F_1$-stable
convergence to a Gaussian limit under additional regularity and
nondegeneracy assumptions and explicit mesh rate conditions.
This stable limit is preserved under equivalent changes of probability
measure on $\mathscr F_1$.
We also give sufficient mesh rate conditions for the pure power-law
kernel.
\section{Notation and assumption}
Throughout the paper, let
$(\Omega,\mathscr F,(\mathscr F_t)_{0\leq t\leq1},P)$
be a filtered probability space satisfying the usual conditions,
and let $W=(W_t)_{0\leq t\leq1}$ be a one-dimensional
$(\mathscr F_t)$-Brownian motion.

Let $q\geq1$, and let $\Theta\subset\mathbb R^q$ be compact and
convex with nonempty interior. The true parameter is denoted by
$\theta_0\in\Theta$.
Let
\begin{equation}
    k:\Theta\times(0,\infty)\longrightarrow(0,\infty)
\end{equation}
be a kernel. We write $k_\theta(x):=k(\theta,x)$.

In each application, $\theta$ denotes the vector of parameters
being estimated. Any fixed model parameters are suppressed
from the notation.

For vectors and matrices, $|\cdot|$ denotes a fixed norm. For $p \geq 1$, we write $|\cdot|_p$ for the $L^p$-norm.

\subsection{Observation scheme}
For each $n\geq1$, observations are taken at
\begin{equation}
    t_i\coloneqq i/n,\qquad i=0,\ldots,n,
    \qquad \Delta_n \coloneqq 1/n.
\end{equation}
For each $d\geq1$, let the maturity grid be
\begin{equation}
    0=T_0<T_1<\cdots<T_d\leq1.
\end{equation}
Write
\begin{equation}
    \delta_j\coloneqq T_j-T_{j-1},
    \qquad
    \Delta_d \coloneqq \max_{1\leq j\leq d}\delta_j,
\end{equation}
and define
\begin{equation}
    \Pi_i
    \coloneqq \{j \mid t_i<T_{j-1}\}.
\end{equation}

Throughout the asymptotic analysis, we assume
\begin{equation}
    n\to\infty,\qquad \Delta_d\to0,\qquad T_d\to1.
\end{equation}
For consistency, no further restriction is imposed on the
relative rates of the time and maturity meshes.
Additional rate conditions for asymptotic normality are
stated in the corresponding theorem.

\begin{rem}
Normalizing the observation period to $[0,1]$ does not require maturities to lie in the same interval. We impose the additional restriction $T_d\leq1$ for simplicity. This choice excludes contracts expiring after the observation period and leads to decreasing maturity coverage towards its end. Maturities beyond the observation period could instead be included by allowing a fixed finite upper bound $T>1$, with corresponding adjustments to the assumptions and limiting expressions introduced below. We retain the present restriction to keep the formulation and proofs concise.
\end{rem}
\subsection{Assumptions on the kernel}
We distinguish the assumptions used for consistency from the
additional assumptions used for asymptotic normality.
The former allow boundary values of the true parameter and
require no differentiability with respect to $\theta$.
The quantitative bounds are imposed only on
$\Theta\times(0,1]$, since the observation horizon is $[0,1]$.

\subsubsection{Assumptions for consistency}

\begin{Khyp}
\label{ass:K0}
The kernel $k$ is continuous on $\Theta\times(0,\infty)$.
For every $\theta\in\Theta$, the function
$x\mapsto k_\theta(x)$ is continuously differentiable, and
$\partial_x k$ is continuous on $\Theta\times(0,\infty)$.
Continuity in $\theta$ is understood in the relative topology
of $\Theta$.
\end{Khyp}

\begin{Khyp}
\label{ass:K1}
There exist constants $c,C>0$ and $\alpha\in(0,1/2]$
such that
\begin{equation}
    k_\theta(x)\geq c,
    \qquad
    |\partial_x\log k_\theta(x)|
    \leq \frac{1/2-\alpha}{x}+C,
    \qquad
    (\theta,x)\in\Theta\times(0,1].
    \label{eq:K1}
\end{equation}
\end{Khyp}

\begin{Khyp}[Identifiability]
\label{ass:K2}
For every $\theta\in\Theta$,
\begin{equation}
    k_\theta(x)=k_{\theta_0}(x)
    \quad\text{for a.e. }x\in(0,1)
    \quad\Longrightarrow\quad
    \theta=\theta_0.
    \label{eq:K2}
\end{equation}
\end{Khyp}

\subsubsection{Additional assumptions for asymptotic normality}
For asymptotic normality, we additionally require
$\theta_0\in\operatorname{Int}(\Theta)$ and impose the following
assumptions on the parameterization under consideration.
These assumptions are not required for consistency.

\begin{Khyp}
\label{ass:K3}
There exists an open set $\mathcal O\subset\mathbb R^q$
containing $\Theta$ such that $k$ admits a positive continuous
extension to $\mathcal O\times(0,\infty)$, denoted by the same
symbol, with continuous partial derivatives
\begin{equation}
    \partial_{\theta_a}k,\qquad
    \partial_{\theta_a\theta_b}k,
    \qquad a,b\in\{1,\ldots,q\},
\end{equation}
on $\mathcal O\times(0,\infty)$.
\end{Khyp}

\begin{Khyp}
\label{ass:K4}
There exist nonincreasing functions
$L_1,L_2:(0,1]\longrightarrow[1,\infty)$ such that
\begin{equation}
    \sup_{0<x\leq1}x^\varepsilon L_\ell(x)<\infty
    \qquad
    \text{for every }\varepsilon>0
    \text{ and }\ell\in\{1,2\},
    \label{eq:L-subpolynomial}
\end{equation}
and there exists a constant $C>0$ such that
\begin{equation}
    |\nabla_\theta\log k_\theta(x)|
    \leq C L_1(x),
    \qquad
    |\nabla_\theta^2\log k_\theta(x)|
    \leq C L_2(x),
    \qquad
    (\theta,x)\in\Theta\times(0,1].
    \label{eq:K4}
\end{equation}
\end{Khyp}

\begin{Khyp}[Local nondegeneracy]
\label{ass:K5}
For every $y\in\mathbb R^q$,
\begin{equation}
    y^\top\nabla_\theta\log k_{\theta_0}(x)=0
    \quad\text{for a.e. }x\in(0,1)
    \quad\Longrightarrow\quad
    y=0.
\end{equation}
\end{Khyp}

\begin{rem}
Compared with \cite{fukasawa2026bergomi}, the present assumptions accommodate
kernels with power-law singularities at the origin,
in addition to regular kernels.
The identifiability and local nondegeneracy requirements
are also weakened.

In particular, the shifted power-law family
\begin{equation}
    k_\theta(x)=\eta(x+\zeta)^{H-1/2},
    \qquad \theta=(\eta,H,\zeta),
\end{equation}
satisfies the consistency assumptions on compact parameter
sets contained in
$(0,\infty)\times(0,1/2)\times[0,\infty)$.
Thus, consistency allows the shift to be estimated jointly
with the other parameters, even when its true value is zero.
When the shift is estimated, however,
Assumption~\ref{ass:K4} fails if the parameter set includes
$\zeta=0$.
For each fixed $\zeta \geq0$, both sets of kernel assumptions
hold for $\theta=(\eta,H)$; no uniformity as $\zeta \downarrow0$
is asserted.
\end{rem}

\subsection{The forward variance model}

\begin{Vhyp}
\label{ass:V0}
The initial forward variance curve
$V_0^\cdot$
is deterministic. There exist constants $\beta \in (0,1]$,
$c,C>0$
such that
\begin{equation}
    \inf_{u\in[0,1]}V_0^u\geq c
    \label{eq:V0-positive}
\end{equation}
and
\begin{equation}
    |V_0^u-V_0^v|
    \leq C|u-v|^\beta,
    \qquad u,v\in[0,1].
    \label{eq:V0-holder}
\end{equation}
\end{Vhyp}

We define the forward variance curve by
\begin{equation}
    V_t^u
    \coloneqq
    V_0^u
    \exp\left(
        \int_0^t k_{\theta_0}(u-s)dW_s
        -\frac12\int_0^t k_{\theta_0}(u-s)^2ds
    \right),
    \qquad 0\leq t\leq u\leq 1.
\end{equation}
Lemma~\ref{OIdfgnoIAMdofigmmrDSK} shows that $(V_t^u)$ admits a modification that is jointly continuous and strictly positive on $\{(t,u)\mid 0\le t\le u\le1\}$ almost surely. We henceforth use this modification, retaining the notation.

\section{Statement of results}
\subsection{Semimartingale representations of the maturity increments}

We first establish the semimartingale representations used in
the subsequent asymptotic analysis.
Define the cumulative forward variance by
\begin{equation}
    I_t^T \coloneqq \int_t^T V_t^udu, \qquad t <T.
    \label{eq:cumulative-forward-variance}
\end{equation}
For $j=2,\ldots,d$ and $0\leq t<T_{j-1}$, define
\begin{equation}
    J_t^j
    \coloneqq I_t^{T_j}-I_t^{T_{j-1}}
    =\int_{T_{j-1}}^{T_j}V_t^u\,du,
    \quad
    \sigma^j(t)
    \coloneqq
    \frac{
        \int_{T_{j-1}}^{T_j}
        k_{\theta_0}(u-t)V_t^u\,du
    }{J_t^j}.
    \label{eq:maturity-increments}
\end{equation}
By Assumption~\ref{ass:K0} and the joint continuity and
strict positivity of the chosen modification of $V$,
almost surely, for every $j=2,\ldots,d$,
the process $\sigma^j$ is well defined and continuous
on $[0,T_{j-1})$, and hence locally bounded.

\begin{pro}
\label{prop:maturity-increment-representations}
Suppose that Assumptions~\ref{ass:K0}, \ref{ass:K1},
and~\ref{ass:V0} hold.
For each fixed maturity grid, the processes $J^j$ are
continuous and strictly positive on their respective
time intervals.
The following identities hold almost surely,
simultaneously for all $j=2,\ldots,d$
and all $0\leq t<T_{j-1}$:
\begin{align}
    J_t^j
    &=
    J_0^j
    +
    \int_0^t
    \left(
        \int_{T_{j-1}}^{T_j}
        k_{\theta_0}(u-s)V_s^udu
    \right)dW_s,
    \label{eq:M-semimartingale}
    \\
    \log J_t^j
    &=
    \log J_0^j
    -\frac12\int_0^t\sigma^j(s)^2ds
    +\int_0^t \sigma^j(s)dW_s.
    \label{eq:log-M-semimartingale}
\end{align}
\end{pro}
\begin{proof}
Fix $j$ and extend $J^j$ to
$[0,T_{j-1}]$ by the same integral definition.
The chosen modification of $V$ makes this extension
continuous and strictly positive almost surely.

The moment bound in Lemma~\ref{OIdfgnoIAMdofigmmrDSK}
and the kernel estimate in
Lemma~\ref{lemA1} give
\begin{align}
    E\left[
        \int_{T_{j-1}}^{T_j}\int_0^{T_{j-1}}
        k_{\theta_0}(u-s)^2(V_s^u)^2\,ds\,du
    \right]
    &\le
    C\int_{T_{j-1}}^{T_j}
    \int_0^{T_{j-1}}
    k_{\theta_0}(u-s)^2dsdu
    \\
    &\le C\delta_jT_{j-1}^{2\alpha}<\infty.
\end{align}
The integrand is jointly progressively measurable,
and this bound verifies the integrability condition
of the stochastic Fubini theorem
\cite[Theorem~2.2]{Veraar01082012}.

For each fixed $u\geq T_{j-1}$, It\^o's formula
applied to the defining exponential gives
\begin{equation}
    V_t^u-V_0^u
    =
    \int_0^t k_{\theta_0}(u-s)V_s^u\,dW_s,
    \qquad 0\leq t\leq T_{j-1},
\end{equation}
as an equality of continuous processes.
Integrating over $u$ and applying stochastic Fubini
yields \eqref{eq:M-semimartingale} as an equality of
continuous processes on $[0,T_{j-1}]$.

Thus $J^j$ is a continuous, strictly positive semimartingale.
It\^o's formula applied to $\log J^j$, together with
the definition of $\sigma^j$, yields
\eqref{eq:log-M-semimartingale} almost surely,
simultaneously for all $0\leq t \leq T_{j-1}$.
A finite intersection over $j$ completes the proof.
\end{proof}

\subsection{Consistency and asymptotic normality}

We estimate $\theta_0$ using a weighted marginal Gaussian
contrast based on the logarithmic increments of $J^j$.
For $j=2,\ldots,d$ and $0\leq t<T_{j-1}$, define
\begin{equation}
    \bar k_\theta^j(t)
    \coloneqq
    \frac{1}{\delta_j}
    \int_{T_{j-1}}^{T_j}k_\theta(u-t)\,du.
    \label{eq:averaged-kernel}
\end{equation}
For $i=1,\ldots,n$ and $j\in\Pi_i$, set
\begin{equation}
    X_{ij}
    \coloneqq
    \log J_{t_i}^j-\log J_{t_{i-1}}^j,\qquad
    q_{ij}(\theta)
    \coloneqq
    \int_{t_{i-1}}^{t_i}
    \bar k_\theta^j(s)^2\,ds.
\end{equation}
The contrast function is
\begin{equation}
    U_{n,d}(\theta)
    \coloneqq
    \frac1n
    \sum_{i=1}^n
    \sum_{j\in\Pi_i}
    \delta_j
    \left[
        \log\left(\frac{q_{ij}(\theta)}{\Delta_n}\right)
        +\frac{X_{ij}^2}{q_{ij}(\theta)}
        +\frac14q_{ij}(\theta)
    \right],
    \qquad \theta\in\Theta.
    \label{eq:contrast}
\end{equation}

For use in the approximation arguments, we also define
\begin{equation}
    \bar X_{ij}
    \coloneqq
    \int_{t_{i-1}}^{t_i}
    \bar k_{\theta_0}^j(s)\,dW_s
    -\frac12q_{ij}(\theta_0).
    \label{eq:gaussian-log-increment}
\end{equation}
Each $\bar X_{ij}$ has distribution
$\mathcal N(-q_{ij}(\theta_0)/2,q_{ij}(\theta_0))$.
Up to terms independent of $\theta$, the contrast
\eqref{eq:contrast} is twice the weighted negative
marginal Gaussian log-likelihood associated with this
approximation.

\begin{thm}[Uniform convergence of the contrast]
\label{thm:uniform-contrast}
Suppose that Assumptions~\ref{ass:K0}, \ref{ass:K1},
and~\ref{ass:V0} hold.
As $n,d\to\infty$,
the contrast $U_{n,d}$ converges uniformly in probability
on $\Theta$ to the deterministic function $U$ defined by
\begin{equation}
    U(\theta)
    \coloneqq
    \int_0^1
    (1-x)
    \left[
        \log\bigl(k_\theta(x)^2\bigr)
        +\frac{k_{\theta_0}(x)^2}{k_\theta(x)^2}
    \right]dx.
    \label{eq:limiting-contrast}
\end{equation}
\end{thm}
No additional restriction on the relative rates of the
time and maturity meshes is required.
The proof of Theorem~\ref{thm:uniform-contrast} is given
in Appendix~\ref{A2}.

Let $\hat\theta_{n,d}$ be any measurable minimizer of
$U_{n,d}$ over $\Theta$.

\begin{cor}[Consistency]
\label{cor:consistency}
In addition to the assumptions of
Theorem~\ref{thm:uniform-contrast}, suppose that
Assumption~\ref{ass:K2} holds.
Then $\hat\theta_{n,d}$ converges in
probability to $\theta_0$ as $n,d\to\infty$.
\end{cor}

\begin{proof}
For each $n,d$, the contrast $U_{n,d}$ is continuous on
the compact set $\Theta$, so a measurable minimizer exists.
The kernel bounds in Assumption~\ref{ass:K1} and dominated
convergence also imply that $U$ is continuous on $\Theta$.

For every $\theta\in\Theta$,
\begin{align}
    U(\theta)-U(\theta_0)
    &=
    \int_0^1(1-x)
    \left[
        \log\left(
            \frac{k_\theta(x)^2}{k_{\theta_0}(x)^2}
        \right)
        +\frac{k_{\theta_0}(x)^2}{k_\theta(x)^2}
        -1
    \right]dx
    \geq0.
    \label{eq:limiting-contrast-identification}
\end{align}
Since $\log z+z^{-1}-1\geq0$ for $z>0$, with equality
if and only if $z=1$, equality in
\eqref{eq:limiting-contrast-identification} implies
$k_\theta=k_{\theta_0}$ almost everywhere on $(0,1)$.
Assumption~\ref{ass:K2} therefore makes $\theta_0$
the unique minimizer of $U$.

By the minimizing property of $\hat\theta_{n,d}$,
\begin{equation}
    0
    \leq
    U(\hat\theta_{n,d})-U(\theta_0)
    \leq
    2\sup_{\theta\in\Theta}
    |U_{n,d}(\theta)-U(\theta)|.
\end{equation}
The right-hand side converges to zero in probability by
Theorem~\ref{thm:uniform-contrast}.
By continuity of $U$, compactness of $\Theta$, and uniqueness
of its minimizer, the preceding bound implies that
$\hat\theta_{n,d}$ converges in probability to $\theta_0$.
\end{proof}

\begin{thm}[Stable asymptotic normality]
\label{thm:stable-asymptotic-normality}
Suppose that the assumptions of
Corollary~\ref{cor:consistency} hold.
Assume additionally that
$\theta_0\in\operatorname{Int}(\Theta)$ and that
Assumptions~\ref{ass:K3}--\ref{ass:K5} hold.
Suppose that, for some
$\varepsilon\in(0,\alpha\wedge\beta)$,
\begin{equation}
    \sqrt n\,
    \Delta_d^{1+\alpha\wedge\beta-\varepsilon}
    \longrightarrow0.
    \label{eq:maximal-mesh-condition}
\end{equation}
Define
\begin{align}
    B
    &\coloneqq
    4\int_0^1
    (1-x)
    \bigl(\nabla_\theta\log k_{\theta_0}(x)\bigr)
    \bigl(\nabla_\theta\log k_{\theta_0}(x)\bigr)^\top
    dx,
    \label{eq:limiting-hessian}
    \\
    D
    &\coloneqq
    8\int_0^1
    \left(
        \int_0^r
        \nabla_\theta\log k_{\theta_0}(x)dx
    \right)
    \left(
        \int_0^r
        \nabla_\theta\log k_{\theta_0}(x)dx
    \right)^\top
    dr,
    \label{eq:limiting-score-covariance}
    \\
    \Gamma
    &\coloneqq B^{-1}DB^{-1}.
    \label{eq:asymptotic-covariance}
\end{align}
Then $B$ and $D$ are positive definite, and $\sqrt n(\hat\theta_{n,d}-\theta_0)$
converges $\mathscr F_1$-stably in law to $Z\sim\mathcal N_q(0,\Gamma)$, as $n,d\to\infty$,
where $Z$ is defined on an extension of the original
probability space and is independent of $\mathscr F_1$.
\end{thm}

The proof of Theorem~\ref{thm:stable-asymptotic-normality}
is given in Appendix~\ref{A3}.

\begin{rem}
\label{rem:power-law-mesh-condition}
Suppose that $T_j=j/d$ and that the true kernel is
$k_{\theta_0}(x)=\eta_0x^{H_0-1/2}$, where
$\eta_0>0$ and $H_0\in(0,1/2)$.
If $d\asymp n^\gamma$, the mesh condition
\eqref{eq:maximal-mesh-condition} in the asymptotic
normality theorem may be replaced by
\begin{equation}
    \gamma>
    \frac{1}{2(1+H_0\wedge\beta)}.
    \label{eq:power-law-grid-exponent}
\end{equation}
Under this condition and the remaining assumptions of
the theorem, its stable asymptotic normality conclusion
holds.
The sharper condition follows by applying the
true-kernel refinement in
Remark~\ref{rem:true-power-law-estimates}
in the proof of Lemma~\ref{lem:score-comparison}.
\end{rem}

\begin{rem}
\label{rem:equivalent-measures}
In the financial interpretation of the model, $P$ is
an equivalent martingale measure under which the forward
variance dynamics are specified.
Let $P^*$ be any fixed probability measure equivalent to
$P$ on $\mathscr F_1$.

Convergence in probability is preserved under this change
of measure. Hence the uniform convergence in
Theorem~\ref{thm:uniform-contrast} and the consistency in
Corollary~\ref{cor:consistency} also hold under $P^*$.

The same argument as in the proof of Proposition~1
of \cite{MyklandZhang2009} shows that the conclusion of
Theorem~\ref{thm:stable-asymptotic-normality}
also holds under $P^*$.
The limiting vector remains independent of $\mathscr F_1$
and, since $\Gamma$ is deterministic, has the same
distribution $\mathcal N_q(0,\Gamma)$.
\end{rem}

\begin{rem}
\label{rem:power-law-asymptotic-covariance}
Consider the two-parameter family
\begin{equation}
    k_\theta(x)=\eta x^{H-1/2},
    \qquad
    \theta=(\eta,H),
    \qquad
    \eta>0,\quad H\in(0,1/2).
\end{equation}
At the true parameter $\theta_0=(\eta_0,H_0)$,
\begin{equation}
    \nabla_\theta\log k_{\theta_0}(x)
    =
    \begin{pmatrix}
        \eta_0^{-1}\\
        \log x
    \end{pmatrix}.
\end{equation}
Substitution into
\eqref{eq:limiting-hessian} and
\eqref{eq:limiting-score-covariance} gives
\begin{equation}
    B
    =
    \begin{pmatrix}
        2/\eta_0^2 & -3/\eta_0\\
        -3/\eta_0 & 7
    \end{pmatrix},\qquad
    D
    =
    \begin{pmatrix}
        8/(3\eta_0^2) & -32/(9\eta_0)\\
        -32/(9\eta_0) & 136/27
    \end{pmatrix}.
\end{equation}
Therefore,
\begin{equation}
    \Gamma
    =
    B^{-1}DB^{-1}
    =
    \begin{pmatrix}
        16\eta_0^2/15 & 8\eta_0/45\\
        8\eta_0/45 & 8/135
    \end{pmatrix}.
    \label{eq:power-law-asymptotic-covariance}
\end{equation}
\end{rem}
\appendix
\section{Auxiliary estimates for the kernel and the forward variance curve}
\label{A1}
Throughout this appendix, assumptions \ref{ass:K0}–\ref{ass:K1} are in force. We collect the kernel comparison, moment, and approximation estimates used in the proofs of consistency and asymptotic normality. The exponents $\alpha$ and $\beta$ are those appearing in \ref{ass:K1} and \ref{ass:V0}, respectively. Additional assumptions, including V0 and \ref{ass:K3}-\ref{ass:K4}, are stated explicitly in the lemmas where they are needed. Neither identifiability \ref{ass:K2} nor local nondegeneracy \ref{ass:K5} is required for these auxiliary estimates.

Throughout the proofs, constants appearing in inequality estimates
are redefined as needed, with $c>0$ and $C >0$ used for lower
and upper bounds, respectively. Unless otherwise stated, these
constants are independent of the parameter, the observation grids,
and the time and maturity variables.

\begin{lem}
\label{lemA1}
    There exist constants $c,C >0$ such that the following assertions hold.
    \begin{enumerate}
        \item For all $\theta \in \Theta$,
        \begin{equation}c \leq k_\theta(x) \leq  Cx^{\alpha-1/2}, \qquad |\partial_x k_\theta(x)| \leq C x^{\alpha-3/2},\qquad x \in (0,1].\end{equation}
        \item
    For all $\theta\in\Theta$ ,
    \begin{equation}
        \int_s^t k_\theta(u-r)^2dr
        \le C|t-s|^{2\alpha},\qquad 0\le s\le t\le u\le1.
    \end{equation}
    Moreover,
    \begin{equation}
        \int_0^t
        |k_\theta(u-r)-k_\theta(v-r)|^2dr
        \le C|u-v|^{2\alpha},\qquad 0\le t\le u\wedge v, \quad u,v \leq 1.
    \end{equation}
        \item For all $\theta \in \Theta$,\ $i$, and  $j \in \Pi_i$,
        \begin{equation}ck_\theta(T_j - t_{i-1})^2 \leq \frac{q_{ij}(\theta)}{\Delta_n} \leq Ck_\theta(T_j-t_{i-1})^2.\end{equation}
        In particular,
        \begin{equation}
            c \Delta_n \leq q_{ij}(\theta) \leq C\Delta_n^{2\alpha}, \qquad \frac{q_{ij}(\theta_0)}{q_{ij}(\theta)} \leq C(T_j - t_{i-1})^{2\alpha -1}.\end{equation}
        Moreover, uniformly in $i$,
            \begin{equation}\sup_{\theta \in \Theta} \sum_{j \in \Pi_i} \delta_j \frac{q_{ij}(\theta)}{\Delta_n} \leq C.
            \end{equation}
    \end{enumerate}
\end{lem}
\begin{proof}
    (i) For any $0 <x \leq y \leq 1$, Assumption \ref{ass:K1} yields
    \begin{equation}|\log k_\theta(x) - \log k_\theta(y)| \leq \int_x^y |\partial_x \log k_\theta(z)|dz \leq \left(\frac12 - \alpha\right) \log \left(\frac yx\right) + C(y-x).\end{equation}
    Consequently, 
    \begin{equation}
    \label{ASOSIGrndfogimOAINOWEIFNOSFI}
        c \left(\frac xy\right)^{1/2-\alpha} \leq \frac{k_\theta(x)}{k_\theta(y)} \leq C \left(\frac{y}{x}\right)^{1/2-\alpha}.
    \end{equation} 
    By \ref{ass:K0} and the compactness of $\Theta$, $\sup_{\theta \in \Theta} k_\theta(1)<\infty$. Taking $y=1$ and combining the resulting upper bound with the lower bound in \ref{ass:K1} proves the first assertion.
    Moreover, $|\partial_x k_\theta(x)| = k_\theta(x)|\partial_x \log k_\theta(x)| \leq Cx^{\alpha - 3/2}$. 

    (ii)
Let $r\ge0$ and $\delta>0$ satisfy $y\coloneqq r +\delta\le1$.
Since $z\le(r+\delta z)/y\le1$ for $z\in(0,1]$,
\eqref{ASOSIGrndfogimOAINOWEIFNOSFI} gives
\begin{equation}
\label{OSIDNFOSIRgnoiNAOSINSORIFGNSODI}
    c z^{1/2-\alpha}
    \le
    \frac{k_\theta(r+\delta z)}{k_\theta(y)}
    \le
    C z^{\alpha-1/2},\qquad z \in (0,1].
\end{equation}
Squaring and integrating, using $\alpha>0$, yields
\begin{equation}
\label{eq:kernel-average}
    c k_\theta(y)^2
    \le
    \int_0^1 k_\theta(r+\delta z)^2dz
    \le
    C k_\theta(y)^2.
\end{equation}

For $0 \leq s<t \leq u \leq 1$, applying \eqref{eq:kernel-average} with
$r=u-t$ and $\delta=t-s$, and then using (i), gives
\begin{equation}
       \int_s^t k_\theta(u-r)^2dr
    \le C(t-s)k_\theta(u-s)^2
    \le C(t-s)(u-s)^{2\alpha-1}
    \le C(t-s)^{2\alpha}, 
\end{equation}
since $u-s\ge t-s$ and $\alpha\le 1/2$.
The case $s=t$ is immediate.

By symmetry
assume $v\le u$ and set $h=u-v$.
The cases $h=0$ and $h=1$ are immediate.
For $0<h<1$ and $0<x\le1-h$, (i) and the
mean value theorem imply
\begin{equation}
    |k_\theta(x+h)-k_\theta(x)|
    \le
    \begin{cases}
        Cx^{\alpha-1/2}, & x<h,\\
        Chx^{\alpha-3/2}, & x\ge h.
    \end{cases}   
\end{equation}
Since $[v-t,v]\subset[0,1-h]$, it follows that
\begin{align}
    \int_0^t
    |k_\theta(u-r)-k_\theta(v-r)|^2\,dr
    &=
    \int_{v-t}^{v}
    |k_\theta(x+h)-k_\theta(x)|^2\,dx \\
    &\le
    C\int_0^h x^{2\alpha-1}\,dx
    +Ch^2\int_h^1 x^{2\alpha-3}\,dx \\
    &\le Ch^{2\alpha}.
\end{align}

(iii)
For $0\le t<T_{j-1}$, applying
\eqref{OSIDNFOSIRgnoiNAOSINSORIFGNSODI}
with $r=T_{j-1}-t$ and $\delta=\delta_j$ yields
\begin{equation}
      \bar k_\theta^j(t)
    \ge
    c k_\theta(T_j-t)
    \int_0^1 z^{1/2-\alpha}\,dz
    \ge c k_\theta(T_j-t).  
\end{equation}
The Cauchy--Schwarz inequality and
\eqref{eq:kernel-average} give the corresponding upper bound.
Thus,
\begin{equation}
\label{WOIEfndorginoIAMOSINfg}
    c k_\theta(T_j-t)
    \le \bar k_\theta^j(t)
    \le C k_\theta(T_j-t).
\end{equation}
Squaring this estimate and integrating over
$(t_{i-1},t_i)$ yields
\begin{equation}
    c\int_{t_{i-1}}^{t_i}k_\theta(T_j-s)^2ds
    \le q_{ij}(\theta)
    \le
    C\int_{t_{i-1}}^{t_i}k_\theta(T_j-s)^2ds.
\end{equation}
Applying \eqref{eq:kernel-average} with
$r=T_j-t_i$, and $\delta=\Delta_n$ gives
\begin{equation}
\label{SDOIGnedorigoIANOISNFOGI}
    c k_\theta(T_j-t_{i-1})^2
    \le
    \frac1{\Delta_n}
    \int_{t_{i-1}}^{t_i}k_\theta(T_j-s)^2ds
    \le
    C k_\theta(T_j-t_{i-1})^2.
\end{equation}
Combining the last two displays proves the comparison
bound for $q_{ij}(\theta)/\Delta_n$.

Since $T_j-t_{i-1}\ge\Delta_n$ and $\alpha\le1/2$,
the preceding comparison and (i) give
    $q_{ij}(\theta)
    \le C\Delta_n(T_j-t_{i-1})^{2\alpha-1}
    \le C\Delta_n^{2\alpha}$.
The lower bound in Assumption~\ref{ass:K1} gives
$q_{ij}(\theta)\ge c\Delta_n$.
Consequently,
\begin{equation}
    \frac{q_{ij}(\theta_0)}{q_{ij}(\theta)}
    \le
    \frac{C\Delta_n(T_j-t_{i-1})^{2\alpha-1}}
         {c\Delta_n}
    \le C(T_j-t_{i-1})^{2\alpha-1}.
\end{equation}

Finally,
\begin{equation}
\label{OISNDogifdfgoiANOISD}
    \sum_{j\in\Pi_i}
    \delta_j\frac{q_{ij}(\theta)}{\Delta_n}
    \le
    C\sum_{j\in\Pi_i}
    \delta_j(T_j-t_{i-1})^{2\alpha-1}.
\end{equation}
Since $2\alpha-1\le0$,
\begin{equation}
    \delta_j(T_j-t_{i-1})^{2\alpha-1}
    \le
    \int_{T_{j-1}}^{T_j}
    (u-t_{i-1})^{2\alpha-1}du.    
\end{equation}
The maturity intervals have disjoint interiors and,
for $j\in\Pi_i$, are contained in $(t_{i-1},1]$.
Consequently,
\begin{equation}
      \sum_{j\in\Pi_i}
    \delta_j(T_j-t_{i-1})^{2\alpha-1}
    \le
    \int_{t_{i-1}}^1
    (u-t_{i-1})^{2\alpha-1}du
    \le C.  
\end{equation}
Combining this with \eqref{OISNDogifdfgoiANOISD}
completes the proof.
\end{proof}

\begin{lem}
\label{lemA2}
    Assume additionally that Assumptions \ref{ass:K3}–\ref{ass:K4} hold. For each $i$ and $j\in\Pi_i$, the function $q_{ij}$, defined using the extension in \ref{ass:K3}, is twice continuously differentiable on $\mathcal O$. Moreover, for every $\varepsilon>0$, there exists a constant $C>0$ such that
\begin{equation}
|\nabla_\theta\log q_{ij}(\theta)|
+|\nabla_\theta^2\log q_{ij}(\theta)|
\le C(T_j-t_{i-1})^{-\varepsilon}
\end{equation}
for all $\theta\in\Theta$, $i$, and $j\in\Pi_i$.
\end{lem}
\begin{proof}
    We first show that $q_{ij}$ is twice continuously differentiable on $\mathcal O$. Fix $i$ and $j\in\Pi_i$. The maturity lags appearing in the defining integrals lie in the compact interval $[T_{j-1}-t_i,T_j-t_{i-1}]\subset(0,1]$.
By Assumption \ref{ass:K3}, $k$ and its parameter derivatives up to order two are continuous and bounded on the product of this interval with any compact neighborhood contained in $\mathcal O$. Differentiation under both integrals is therefore justified twice, and the resulting derivatives are continuous in $\theta$. Hence $q_{ij}$ is twice continuously differentiable on $\mathcal O$.

Specifically,
\begin{equation}
    \nabla_\theta\bar k_\theta^j(t)
=
\frac{1}{\delta_j}
\int_{T_{j-1}}^{T_j}
\nabla_\theta k_\theta(u-t)du,
\qquad
\nabla_\theta^2\bar k_\theta^j(t)
=
\frac{1}{\delta_j}
\int_{T_{j-1}}^{T_j}
\nabla_\theta^2k_\theta(u-t)du,
\end{equation}
and consequently,
\begin{align}
\label{WSDFOISnfsdgoiNAOisfgndfgi}
    \nabla_\theta q_{ij}(\theta)
&=
2\int_{t_{i-1}}^{t_i}
\bar k_\theta^j(s)\nabla_\theta\bar k_\theta^j(s)ds,\\
\label{SOdifndfogimoIANSOidno}
\nabla_\theta^2q_{ij}(\theta)
&=
2\int_{t_{i-1}}^{t_i}
\left[
\nabla_\theta\bar k_\theta^j(s)
\nabla_\theta\bar k_\theta^j(s)^\top
+
\bar k_\theta^j(s)\nabla_\theta^2\bar k_\theta^j(s)
\right]ds.
\end{align}

Next, fix $0<\varepsilon<2\alpha$. Since
\begin{align}
\nabla_\theta k_\theta(x)
&=
k_\theta(x)\nabla_\theta\log k_\theta(x),\\
\nabla_\theta^2 k_\theta(x)
&=
k_\theta(x)\left[
\nabla_\theta^2\log k_\theta(x)
+
\nabla_\theta\log k_\theta(x)
\nabla_\theta\log k_\theta(x)^\top
\right],
\end{align}
Assumption \ref{ass:K4} implies that
\begin{equation}
    |\nabla_\theta k_\theta(x)|
+
|\nabla_\theta^2 k_\theta(x)|
\le
C k_\theta(x)x^{-\varepsilon/2},
\qquad
(\theta,x)\in\Theta\times(0,1].
\end{equation}
Fix $t<T_{j-1}$. The preceding derivative bounds imply that
\begin{equation}
|\nabla_\theta\bar k_\theta^j(t)|
+
|\nabla_\theta^2\bar k_\theta^j(t)|
\le
\frac{C}{\delta_j}
\int_{T_{j-1}}^{T_j}
k_\theta(u-t)(u-t)^{-\varepsilon/2}\,du.
\end{equation}
To estimate the integral, applying \eqref{OSIDNFOSIRgnoiNAOSINSORIFGNSODI} with $r=T_{j-1}-t$ and $\delta=\delta_j$, and noting that
\begin{equation}
T_{j-1}-t+\delta_jz
\ge (T_j-t)z,
\qquad 0<z\le1,
\end{equation}
we obtain
\begin{align}
\frac{1}{\delta_j}
\int_{T_{j-1}}^{T_j}
k_\theta(u-t)(u-t)^{-\varepsilon/2}\,du
&=
\int_0^1
k_\theta(T_{j-1}-t+\delta_jz)
(T_{j-1}-t+\delta_jz)^{-\varepsilon/2}\,dz
\\
&\leq
C k_\theta(T_j-t)(T_j-t)^{-\varepsilon/2}
\int_0^1 z^{\alpha-1/2-\varepsilon/2}\,dz
\end{align}
The last integral is finite because $\varepsilon<2\alpha$. Combining this estimate with the lower comparison bound in \eqref{WOIEfndorginoIAMOSINfg} yields
\begin{equation}
\label{SODifndofgimoAINOSIDnfoi}
|\nabla_\theta\bar k_\theta^j(t)|
+
|\nabla_\theta^2\bar k_\theta^j(t)|
\le
C\bar k_\theta^j(t)(T_j-t)^{-\varepsilon/2}.
\end{equation}

Substituting \eqref{SODifndofgimoAINOSIDnfoi} into
\eqref{WSDFOISnfsdgoiNAOisfgndfgi} and \eqref{SOdifndfogimoIANSOidno} gives
\begin{equation}
\label{SODIfdfgmOSIFNgodfigmZO}
|\nabla_\theta q_{ij}(\theta)|
\le
C
\int_{t_{i-1}}^{t_i}
\bar k_\theta^j(s)^2(T_j-s)^{-\varepsilon/2}ds, \quad
|\nabla_\theta^2 q_{ij}(\theta)|\le
C
\int_{t_{i-1}}^{t_i}
\bar k_\theta^j(s)^2(T_j-s)^{-\varepsilon}ds.
\end{equation}
We now estimate the weighted time integral. By the upper comparison
bound in \eqref{WOIEfndorginoIAMOSINfg},
\begin{equation}
\int_{t_{i-1}}^{t_i}
\bar k_\theta^j(s)^2(T_j-s)^{-\varepsilon}ds
\le
C\int_{t_{i-1}}^{t_i}
k_\theta(T_j-s)^2(T_j-s)^{-\varepsilon}ds.
\end{equation}
Applying \eqref{OSIDNFOSIRgnoiNAOSINSORIFGNSODI} with
$r=T_j-t_i$ and $\delta=\Delta_n$,  we obtain
\begin{align}
\int_{t_{i-1}}^{t_i}
k_\theta(T_j-s)^2(T_j-s)^{-\varepsilon}ds&=
\Delta_n\int_0^1
k_\theta(T_j-t_i+\Delta_n z)^2
(T_j-t_i+\Delta_n z)^{-\varepsilon}dz\\
&\le
C\Delta_n k_\theta(T_j-t_{i-1})^2
(T_j-t_{i-1})^{-\varepsilon}
\int_0^1 z^{2\alpha-1-\varepsilon}dz.
\end{align}
Here we also used
$T_j-t_i+\Delta_n z\ge (T_j-t_{i-1})z$.
The last integral is finite since $\varepsilon<2\alpha$.
Combining this estimate with the lower bound
from Lemma \ref{lemA1} yields
\begin{equation}
\label{SOdifgndfogimnzodicSB}
\int_{t_{i-1}}^{t_i}
\bar k_\theta^j(s)^2(T_j-s)^{-\varepsilon}ds
\le
C q_{ij}(\theta)
(T_j-t_{i-1})^{-\varepsilon}.
\end{equation}
The same argument proves \eqref{SOdifgndfogimnzodicSB}
with $\varepsilon/2$ in place of $\varepsilon$.
Applying these two bounds to
\eqref{SODIfdfgmOSIFNgodfigmZO} yields
\begin{equation}
\label{ERgoinAOInsdinZOImN}
|\nabla_\theta q_{ij}(\theta)|
\le
C q_{ij}(\theta)
(T_j-t_{i-1})^{-\varepsilon/2}, \quad
|\nabla_\theta^2 q_{ij}(\theta)|
\le
C q_{ij}(\theta)
(T_j-t_{i-1})^{-\varepsilon}.    
\end{equation}

Since $q_{ij}>0$, we have
\begin{align}
\nabla_\theta\log q_{ij}(\theta)
&=
\frac{\nabla_\theta q_{ij}(\theta)}{q_{ij}(\theta)},\\
\nabla_\theta^2\log q_{ij}(\theta)
&=
\frac{\nabla_\theta^2q_{ij}(\theta)}{q_{ij}(\theta)}
-
\frac{
\nabla_\theta q_{ij}(\theta)
\nabla_\theta q_{ij}(\theta)^\top
}{
q_{ij}(\theta)^2
}.
\end{align}
Substituting \eqref{ERgoinAOInsdinZOImN} into these
identities gives
\begin{equation}
    |\nabla_\theta\log q_{ij}(\theta)|+
|\nabla_\theta^2\log q_{ij}(\theta)|\leq 
C(T_j-t_{i-1})^{-\varepsilon},\qquad 0 <\varepsilon < 2\alpha.
\end{equation}

Finally, if $\varepsilon\ge2\alpha$, apply the estimate already proved
with exponent $\alpha$, which belongs to $(0,2\alpha)$, and use
$(T_j-t_{i-1})^{-\alpha}
\le (T_j-t_{i-1})^{-\varepsilon}$. This completes the proof.
\end{proof}

\begin{lem}
\label{OIdfgnoIAMdofigmmrDSK}
    Assume additionally that Assumption \ref{ass:V0} holds. For every $p \geq 1$, there exists a constant $C >0$ such that
    \begin{equation}
        \sup_{0\le t\le u\le1}
E\left[(V_t^u)^p+(V_t^u)^{-p}\right]
\le C
    \end{equation}
    and
    \begin{equation}
        |V_t^u-V_s^v|_p
\le C\left(
|t-s|^\alpha+|u-v|^{\alpha\wedge\beta}
\right),
\qquad
0\le t\le u\le1,\quad 0\le s\le v\le1.
    \end{equation}
    Moreover, the field $(V_t^u)_{0\leq t \leq u \leq 1}$ admits a modification that is jointly continuous and strictly positive  almost surely.
\end{lem}
\begin{proof}
For every $p\geq1$, the definition of $V_t^u$ gives
\begin{align}
E[(V_t^u)^p]
&=(V_0^u)^p
\exp\left\{
\frac{p^2-p}{2}\int_0^t k_{\theta_0}(u-r)^2\,dr
\right\}, \\
E[(V_t^u)^{-p}]
&=(V_0^u)^{-p}
\exp\left\{
\frac{p^2+p}{2}\int_0^t k_{\theta_0}(u-r)^2\,dr
\right\}.
\end{align}
Assumption \ref{ass:V0} and Lemma~\ref{lemA1} therefore imply
\begin{equation}
\sup_{0\le t\le u\le1}
E\!\left[(V_t^u)^p+(V_t^u)^{-p}\right]\le C.
\label{eq:fv-moments}
\end{equation}

For $0\le s\le t\le u\le1$, we have
\begin{equation}
\log V_t^u-\log V_s^u
=
\int_s^t k_{\theta_0}(u-r)dW_r
-\frac12\int_s^t k_{\theta_0}(u-r)^2dr.
\end{equation}
The Gaussian moment formula and the kernel estimates in Lemma~\ref{lemA1} yield
\begin{align}
|\log V_t^u-\log V_s^u|_p
&\le
C\left(\int_s^t k_{\theta_0}(u-r)^2dr\right)^{1/2}
+\frac12\int_s^t k_{\theta_0}(u-r)^2dr \notag\\
&\le C\bigl(|t-s|^\alpha+|t-s|^{2\alpha}\bigr)
\le C|t-s|^\alpha,
\label{eq:fv-log-time}
\end{align}
where the last inequality uses $|t-s|\le1$ and $\alpha>0$.

Next, let $t\le u\wedge v$. Then
\begin{align}
\log V_t^u-\log V_t^v
&=
\log V_0^u-\log V_0^v
+\int_0^t
\bigl(k_{\theta_0}(u-r)-k_{\theta_0}(v-r)\bigr)dW_r \\
&\quad
-\frac12\int_0^t
\bigl(k_{\theta_0}(u-r)^2-k_{\theta_0}(v-r)^2\bigr)dr.
\end{align}
The lower bound and H\"{o}lder continuity in \ref{ass:V0} give
$|\log V_0^u-\log V_0^v|
\le C|u-v|^\beta$.
By the Gaussian moment formula and Lemma~\ref{lemA1},
\begin{align}
    \left|
\int_0^t
\bigl(k_{\theta_0}(u-r)-k_{\theta_0}(v-r)\bigr)dW_r
\right|_p &\le C\left(
\int_0^t
|k_{\theta_0}(u-r)-k_{\theta_0}(v-r)|^2dr
\right)^{1/2}\\
&\le C|u-v|^\alpha.
\end{align}
For the drift term, Cauchy–Schwarz and the same kernel estimates imply
\begin{equation}
    \left|
\int_0^t
\bigl(k_{\theta_0}(u-r)^2-k_{\theta_0}(v-r)^2\bigr)dr
\right|
\le C|u-v|^\alpha.
\end{equation}    
Combining these bounds gives
\begin{equation}
|\log V_t^u-\log V_t^v|_p
\le C\bigl(|u-v|^\beta+|u-v|^\alpha\bigr)
\le C|u-v|^{\alpha\wedge\beta},
\label{eq:fv-log-maturity}
\end{equation}
since $|u-v|\le1$.

Now take arbitrary points satisfying $0\le t\le u\le1$ and $0\le s\le v\le1$. By symmetry, we may assume $s\le t$, in which case $s\le u\wedge v$. The triangle inequality, \eqref{eq:fv-log-time}, and \eqref{eq:fv-log-maturity} give
\begin{align}
|\log V_t^u-\log V_s^v|_p
&\le
|\log V_t^u-\log V_s^u|_p
+|\log V_s^u-\log V_s^v|_p \\
&\le
C\left(|t-s|^\alpha+|u-v|^{\alpha\wedge\beta}\right).
\label{eq:fv-log-joint}
\end{align}
Moreover, using H\"{o}lder's inequality, together with \eqref{eq:fv-moments} and \eqref{eq:fv-log-joint}, we obtain
\begin{align}
|V_t^u-V_s^v|_{p}
&\le
|\log V_t^u-\log V_s^v|_{2p}
\left(|V_t^u|_{2p}+|V_s^v|_{2p}\right)\\
&\le
C\left(|t-s|^\alpha+|u-v|^{\alpha\wedge\beta}\right).
\end{align}

Finally, choose $p$ such that $p(\alpha\wedge\beta)>2$. 
\eqref{eq:fv-log-joint} and the Kolmogorov–Chentsov theorem yield a jointly continuous modification of $(t,u)\mapsto\log V_t^u$ on the closed triangle $\{(t,u) \mid 0\le t\le u\le1\}$. Exponentiating this modification gives a modification of $(V_t^u)_{0 \leq t \leq u \leq 1}$ which, on an event of probability one, is jointly continuous and strictly positive throughout the triangle.
\end{proof}

\begin{lem}
\label{SODIgndfogimOAIENWOrig}
    Assume additionally that Assumption \ref{ass:V0} holds. For any $p \geq 1$, there exists a constant $C$ such that 
    \begin{equation}
        \left|\frac{\sigma^j(t) - \bar k_{\theta_0}^j(t)}{\bar k_{\theta_0}^j(t)}\right|_p \leq C \frac{\delta_j^{1+\alpha \wedge \beta}}{T_j-t}, \qquad t <T_{j-1}.
    \end{equation}
    In particular,
    \begin{equation}
        |\sigma^j(t) - \bar k_{\theta_0}^j (t)|_p \leq C \delta_j^{1+\alpha \wedge \beta}(T_j-t)^{\alpha-3/2}, \qquad t<T_{j-1}.
    \end{equation}
\end{lem}
\begin{proof}
Fix $p\ge1$ and $t<T_{j-1}$. By the definitions of $\sigma^j(t)$ and
$\bar k_{\theta_0}^j(t)$,
\begin{equation}
\label{eq:sigma-centered-identity}
    \sigma^j(t)-\bar k_{\theta_0}^j(t)
    =
    \frac{1}{J_t^j}
    \int_{T_{j-1}}^{T_j}
    (V_t^u-V_t^{T_j})
    \bigl(k_{\theta_0}(u-t)-\bar k_{\theta_0}^j(t)\bigr)\,du.
\end{equation}
Applying Minkowski's integral inequality to
\eqref{eq:sigma-centered-identity} gives
\begin{equation}
\label{eq:sigma-minkowski}
    \left|
    \frac{\sigma^j(t)-\bar k_{\theta_0}^j(t)}
         {\bar k_{\theta_0}^j(t)}
    \right|_p
    \le
    \frac{1}{\delta_j\bar k_{\theta_0}^j(t)}
    \int_{T_{j-1}}^{T_j}
    \bigl|k_{\theta_0}(u-t)-\bar k_{\theta_0}^j(t)\bigr|
    \left|
        \frac{\delta_j}{J_t^j}(V_t^u-V_t^{T_j})
    \right|_pdu.
\end{equation}

Jensen's inequality and the negative moment bound in
Lemma~\ref{OIdfgnoIAMdofigmmrDSK} yield
\begin{equation}
\label{eq:sigma-inverse-moment}
    E\left[\left(\frac{\delta_j}{J_t^j}\right)^{2p}\right]
    =
    E\left[
        \left(
            \frac{1}{\delta_j}
            \int_{T_{j-1}}^{T_j}V_t^udu
        \right)^{-2p}
    \right] \le
    \frac{1}{\delta_j}
    \int_{T_{j-1}}^{T_j}E[(V_t^u)^{-2p}]du
    \le C.    
\end{equation}
Thus, by H\"older's inequality,
\eqref{eq:sigma-inverse-moment}, and the increment bound in
Lemma~\ref{OIdfgnoIAMdofigmmrDSK}, for every $u\in[T_{j-1},T_j]$,
\begin{equation}
    \left|
        \frac{\delta_j}{J_t^j}(V_t^u-V_t^{T_j})
    \right|_p\le
    \left|\frac{\delta_j}{J_t^j}\right|_{2p}
    |V_t^u-V_t^{T_j}|_{2p} \le C|T_j-u|^{\alpha\wedge\beta}
    \le C\delta_j^{\alpha\wedge\beta}.
\end{equation}
Substituting this estimate into \eqref{eq:sigma-minkowski},
we obtain
\begin{equation}
\label{eq:sigma-error-reduction}
    \left|
    \frac{\sigma^j(t)-\bar k_{\theta_0}^j(t)}
         {\bar k_{\theta_0}^j(t)}
    \right|_p
    \le
    \frac{C\delta_j^{\alpha\wedge\beta}}
         {\delta_j\bar k_{\theta_0}^j(t)}
    \int_{T_{j-1}}^{T_j}
    \bigl|k_{\theta_0}(u-t)-\bar k_{\theta_0}^j(t)\bigr|du.
\end{equation}

We next estimate the deterministic integral in
\eqref{eq:sigma-error-reduction}. If
$T_{j-1}-t<\delta_j$, positivity and the definition of
$\bar k_{\theta_0}^j(t)$ imply
\begin{align}
\label{eq:sigma-kernel-near}
    &\frac{1}{\delta_j\bar k_{\theta_0}^j(t)}
    \int_{T_{j-1}}^{T_j}
    \bigl|k_{\theta_0}(u-t)-\bar k_{\theta_0}^j(t)\bigr|\,du
    \notag\\
    &\qquad\le
    \frac{1}{\delta_j\bar k_{\theta_0}^j(t)}
    \int_{T_{j-1}}^{T_j}
    \bigl(k_{\theta_0}(u-t)+\bar k_{\theta_0}^j(t)\bigr)\,du
    =2
    \le\frac{4\delta_j}{T_j-t},
\end{align}
where the last inequality follows from
$T_j-t=(T_{j-1}-t)+\delta_j<2\delta_j$.

Suppose instead that $T_{j-1}-t\ge\delta_j$.
Applying \eqref{ASOSIGrndfogimOAINOWEIFNOSFI} with
$y=T_j-t$ gives
\begin{equation}
\label{eq:sigma-kernel-comparison}
    k_{\theta_0}(x)
    \le
    C\left(\frac{T_j-t}{x}\right)^{1/2-\alpha}
    k_{\theta_0}(T_j-t)
    \le Ck_{\theta_0}(T_j-t),
    \qquad
    x\in[T_{j-1}-t,T_j-t],
\end{equation}
since $T_j-t\le2(T_{j-1}-t)\le2x$.
Combining \eqref{eq:sigma-kernel-comparison} with
Assumption~\ref{ass:K1}, we obtain
\begin{equation}
    |\partial_x k_{\theta_0}(x)|
    =
    k_{\theta_0}(x)
    |\partial_x\log k_{\theta_0}(x)| 
    \le
    C\frac{k_{\theta_0}(T_j-t)}{T_{j-1}-t}, 
\end{equation}
where we used $x\ge T_{j-1}-t$ and $T_{j-1}-t\le1$.
The mean value theorem therefore yields
\begin{equation}
\label{eq:sigma-kernel-increment}
    |k_{\theta_0}(u-t)-k_{\theta_0}(v-t)|\le
    \frac{C|u-v|}{T_{j-1}-t}k_{\theta_0}(T_j-t)\le
    \frac{C\delta_j}{T_j-t}k_{\theta_0}(T_j-t),
    \quad u,v\in[T_{j-1},T_j],    
\end{equation}
again using $T_j-t\le2(T_{j-1}-t)$.

Averaging \eqref{eq:sigma-kernel-increment} with respect
to $v$ gives
\begin{equation}
    |k_{\theta_0}(u-t)-\bar k_{\theta_0}^j(t)|\le
    \frac{1}{\delta_j}
    \int_{T_{j-1}}^{T_j}
    |k_{\theta_0}(u-t)-k_{\theta_0}(v-t)|dv\le
    C\frac{\delta_j}{T_j-t}k_{\theta_0}(T_j-t).
\end{equation}
Integrating this estimate with respect to $u$ and using the
lower bound in \eqref{WOIEfndorginoIAMOSINfg}, we conclude that
\begin{equation}
    \label{eq:sigma-kernel-far}
    \frac{1}{\delta_j\bar k_{\theta_0}^j(t)}
    \int_{T_{j-1}}^{T_j}
    |k_{\theta_0}(u-t)-\bar k_{\theta_0}^j(t)|\,du
  \le
    C\frac{\delta_j}{T_j-t}
    \frac{k_{\theta_0}(T_j-t)}{\bar k_{\theta_0}^j(t)}
    \le C\frac{\delta_j}{T_j-t}.
\end{equation}

Combining \eqref{eq:sigma-kernel-near} and
\eqref{eq:sigma-kernel-far} with
\eqref{eq:sigma-error-reduction} proves
\begin{equation}
\label{eq:sigma-relative-error}
    \left|
    \frac{\sigma^j(t)-\bar k_{\theta_0}^j(t)}
         {\bar k_{\theta_0}^j(t)}
    \right|_p
    \le
    C\frac{\delta_j^{1+\alpha\wedge\beta}}{T_j-t}.
\end{equation}

Finally, the upper bound in
\eqref{WOIEfndorginoIAMOSINfg}, together with
$k_{\theta_0}(x)\le Cx^{\alpha-1/2}$, gives
\begin{equation}
    \bar k_{\theta_0}^j(t)
    \le Ck_{\theta_0}(T_j-t)
    \le C(T_j-t)^{\alpha-1/2}.
\end{equation}
Multiplying \eqref{eq:sigma-relative-error} by this
deterministic bound yields
\begin{equation}
    |\sigma^j(t)-\bar k_{\theta_0}^j(t)|_p
    \le
    C\delta_j^{1+\alpha\wedge\beta}
    (T_j-t)^{\alpha-3/2},
\end{equation}
which completes the proof.
\end{proof}

\begin{lem}
\label{lemA5}
    Assume additionally that Assumption \ref{ass:V0} holds. For every $p\ge1$, there exists a constant $C>0$ such that
\begin{equation}
    |X_{ij}|_p+|\bar X_{ij}|_p
\le C\sqrt{q_{ij}(\theta_0)}
\end{equation}
and
\begin{equation}
   |X_{ij}-\bar X_{ij}|_p
\le
C\sqrt{q_{ij}(\theta_0)}
\frac{\delta_j^{1+\alpha\wedge\beta}}{T_j-t_i} 
\end{equation}
for all $i$ and $j\in\Pi_i$. Consequently,
\begin{equation}
   \left|
\frac{X_{ij}^2-\bar X_{ij}^2}{q_{ij}(\theta_0)}
\right|_p
\le
C\frac{\delta_j^{1+\alpha\wedge\beta}}{T_j-t_i} 
\end{equation}
for the same indices.
\end{lem}
\begin{proof}
Fix $i$, $j\in\Pi_i$, and $p\ge2$, and write
$q=q_{ij}(\theta_0)$ throughout this proof.
By Lemma~\ref{SODIgndfogimOAIENWOrig}, for any $p' \geq 1$ and
$s\in[t_{i-1},t_i]$,
\begin{equation}
\label{eq:X-coefficient-error}
|\sigma^j(s)-\bar k_{\theta_0}^j(s)|_{p'}
\le
C\bar k_{\theta_0}^j(s)
\frac{\delta_j^{1+\alpha\wedge\beta}}{T_j-t_i}.
\end{equation}
Since $T_j-t_i>\delta_j$ and $\delta_j\le1$, the triangle
inequality yields
\begin{equation}
\label{eq:X-coefficient-moment}
|\sigma^j(s)|_{p'}\le C\bar k_{\theta_0}^j(s).
\end{equation}

The Burkholder--Davis--Gundy and Minkowski inequalities,
together with \eqref{eq:X-coefficient-moment}, give
\begin{equation}
\left|
\int_{t_{i-1}}^{t_i}\sigma^j(s)dW_s
\right|_p
\le
C\left|
\int_{t_{i-1}}^{t_i}\sigma^j(s)^2ds
\right|_{p/2}^{1/2}\le
C\left(
\int_{t_{i-1}}^{t_i}|\sigma^j(s)|_p^2ds
\right)^{1/2}
\le C\sqrt q.    
\end{equation}
Similarly,
\begin{equation}
\left|
\int_{t_{i-1}}^{t_i}\sigma^j(s)^2ds
\right|_p
\le
\int_{t_{i-1}}^{t_i}|\sigma^j(s)|_{2p}^2ds
\le Cq.
\end{equation}
Thus $|X_{ij}|_p\le C(\sqrt q+q)$.
Since $\bar X_{ij}\sim N(-q/2,q)$, we also have
$|\bar X_{ij}|_p\le C\sqrt q+q/2$.
Lemma~\ref{lemA1} gives $q\le C\Delta_n^{2\alpha}\le C$,
and hence $q\le C\sqrt q$. Consequently,
\begin{equation}
\label{eq:X-moments}
|X_{ij}|_p+|\bar X_{ij}|_p\le C\sqrt q.
\end{equation}

Next, write
\begin{equation}
    \label{eq:X-difference}
X_{ij}-\bar X_{ij}
=\int_{t_{i-1}}^{t_i}
\bigl(\sigma^j(s)-\bar k_{\theta_0}^j(s)\bigr)dW_s
-\frac12\int_{t_{i-1}}^{t_i}
\bigl(\sigma^j(s)^2-\bar k_{\theta_0}^j(s)^2\bigr)ds.
\end{equation}
Applying the Burkholder--Davis--Gundy and Minkowski
inequalities as above and using
\eqref{eq:X-coefficient-error}, we obtain
\begin{align}
\left|
\int_{t_{i-1}}^{t_i}
\bigl(\sigma^j(s)-\bar k_{\theta_0}^j(s)\bigr)dW_s
\right|_p&\le
C\left(
\int_{t_{i-1}}^{t_i}
|\sigma^j(s)-\bar k_{\theta_0}^j(s)|_p^2ds
\right)^{1/2}\\
&\le
C\sqrt q\,
\frac{\delta_j^{1+\alpha\wedge\beta}}{T_j-t_i}.
\end{align}
For the drift term, H\"older's inequality and
\eqref{eq:X-coefficient-error}--\eqref{eq:X-coefficient-moment}
yield
\begin{equation}
|\sigma^j(s)^2-\bar k_{\theta_0}^j(s)^2|_p
\le
|\sigma^j(s)-\bar k_{\theta_0}^j(s)|_{2p}
|\sigma^j(s)+\bar k_{\theta_0}^j(s)|_{2p}
\le
C\bar k_{\theta_0}^j(s)^2
\frac{\delta_j^{1+\alpha\wedge\beta}}{T_j-t_i}.    
\end{equation}
Therefore, Minkowski's inequality gives
\begin{equation}
    \left|
\int_{t_{i-1}}^{t_i}
\bigl(\sigma^j(s)^2-\bar k_{\theta_0}^j(s)^2\bigr)ds
\right|_p
\le
\int_{t_{i-1}}^{t_i}
|\sigma^j(s)^2-\bar k_{\theta_0}^j(s)^2|_pds
\le
Cq\,
\frac{\delta_j^{1+\alpha\wedge\beta}}{T_j-t_i}.
\end{equation}
Combining these estimates with \eqref{eq:X-difference}
and $q\le C\sqrt q$ proves
\begin{equation}
\label{eq:X-difference-moment}
|X_{ij}-\bar X_{ij}|_p
\le
C\sqrt q\,
\frac{\delta_j^{1+\alpha\wedge\beta}}{T_j-t_i}.
\end{equation}
The bounds \eqref{eq:X-moments} and
\eqref{eq:X-difference-moment} extend to $1\le p<2$
by the inequality $|\cdot|_p\le|\cdot|_2$.

Finally, for any $p\ge1$, H\"older's inequality and
\eqref{eq:X-moments}--\eqref{eq:X-difference-moment},
applied with exponent $2p$, imply
\begin{equation}
\left|
\frac{X_{ij}^2-\bar X_{ij}^2}{q}
\right|_p
\le
\frac1q
|X_{ij}-\bar X_{ij}|_{2p}
\bigl(|X_{ij}|_{2p}+|\bar X_{ij}|_{2p}\bigr)
\le
C\frac{\delta_j^{1+\alpha\wedge\beta}}{T_j-t_i}.    
\end{equation}
This completes the proof.
\end{proof}

\begin{rem}
\label{rem:true-power-law-estimates}
The exponent $\alpha$ in Assumption~\ref{ass:K1}
provides a bound that is uniform over the parameter set.
For estimates involving only the true kernel and the
forward variance curve generated by it, this exponent
can be sharpened in the power-law case.
Specifically, suppose that
\begin{equation}
    k_{\theta_0}(x)=\eta_0x^{H_0-1/2},
    \qquad \eta_0>0,\quad H_0\in(0,1/2).
\end{equation}
Applying the proof of Lemma~\ref{lemA5} with $H_0$
in place of $\alpha$, we obtain, for every $p\ge1$,
\begin{equation}
    \left|
        \frac{X_{ij}^2-\bar X_{ij}^2}
             {q_{ij}(\theta_0)}
    \right|_p
    \le
    C\frac{\delta_j^{1+H_0\wedge\beta}}{T_j-t_i},
    \qquad j\in\Pi_i.
\end{equation}

The constants may depend on $p$ and $\theta_0$, but remain
independent of the observation grids.
This refinement concerns the fixed true parameter and
does not assert that Assumption~\ref{ass:K1} holds with
$\alpha=H_0$ uniformly over $\Theta$.
The exponent $\beta$ remains the one specified in
Assumption~\ref{ass:V0}.
\end{rem}
\section{Proof of consistency}
\label{A2}
Throughout this appendix, we assume
\ref{ass:K0}--\ref{ass:K2} and \ref{ass:V0}.
Let $\bar U_{n,d}$ denote the criterion obtained from
$U_{n,d}$ by replacing $X_{ij}$ with its Gaussian
approximation $\bar X_{ij}$:
\begin{equation}
\label{barU}
    \bar U_{n,d}(\theta)
    \coloneqq
    \frac{1}{n}
    \sum_{i=1}^{n}\sum_{j\in\Pi_i}\delta_j
    \left[
        \log\frac{q_{ij}(\theta)}{\Delta_n}
        +\frac{\bar X_{ij}^{2}}{q_{ij}(\theta)}
        +\frac{q_{ij}(\theta)}{4}
    \right].
\end{equation}

Throughout the proofs, $c >0$ and $C>0$ denote generic constants
in lower
and upper bounds, respectively. Their values may change from line to line.
Unless otherwise stated, these constants are independent of the parameter,
the observation grids, and the time and maturity variables.

\begin{lem}
    \label{lem:criterion-gaussian-approximation}
There exists a constant $C>0$ such that
\begin{equation}
    E\left[
        \sup_{\theta\in\Theta}
        \left|U_{n,d}(\theta)-\bar U_{n,d}(\theta)\right|
    \right]
    \le C\Delta_d^{\,\alpha\wedge\beta}
\end{equation}
for all $n,d$.
\end{lem}
\begin{proof}
By definition,
\begin{equation}
U_{n,d}(\theta)-\bar U_{n,d}(\theta)
=
\frac{1}{n}\sum_{i=1}^n\sum_{j\in\Pi_i}
\delta_j
\frac{X_{ij}^2-\bar X_{ij}^2}{q_{ij}(\theta)}.
\end{equation}
Since Lemma~\ref{lemA1} gives
$\inf_{\theta\in\Theta}q_{ij}(\theta)\ge c\Delta_n$,
the triangle inequality yields
\begin{equation}
\label{eq:contrast-replacement-sup}
\sup_{\theta\in\Theta}
|U_{n,d}(\theta)-\bar U_{n,d}(\theta)|
\le
\frac{C}{n}\sum_{i=1}^n\sum_{j\in\Pi_i}
\frac{\delta_j}{\Delta_n}
|X_{ij}^2-\bar X_{ij}^2|.
\end{equation}

For $j\in\Pi_i$, we have $t_i<T_{j-1}$ and hence
$T_j-t_i>\delta_j$. Applying Lemma~\ref{lemA5} and using
$\delta_j\le\Delta_d$, we obtain
\begin{equation}
\label{eq:contrast-replacement-increment}
E|X_{ij}^2-\bar X_{ij}^2|
\le
Cq_{ij}(\theta_0)
\frac{\delta_j^{1+(\alpha\wedge\beta)}}{T_j-t_i}\le
Cq_{ij}(\theta_0)\Delta_d^{\alpha\wedge\beta}.
\end{equation}
Moreover, Lemma~\ref{lemA1} gives, uniformly in $i$,
$\sum_{j\in\Pi_i}
\delta_jq_{ij}(\theta_0)/\Delta_n
\le C$.
Taking expectations in
\eqref{eq:contrast-replacement-sup} and applying
\eqref{eq:contrast-replacement-increment} together with
this uniform bound, we conclude that
\begin{equation}
E\left[
\sup_{\theta\in\Theta}
|U_{n,d}(\theta)-\bar U_{n,d}(\theta)|
\right]\le
\frac{C\Delta_d^{\alpha\wedge\beta}}{n}
\sum_{i=1}^n\sum_{j\in\Pi_i}
\delta_j\frac{q_{ij}(\theta_0)}{\Delta_n}\le
C\Delta_d^{\alpha\wedge\beta}.
\end{equation}
This completes the proof.
\end{proof}

\begin{lem}
    As $n,d\to\infty$, $\bar U_{n,d}$ converges to $U$
    uniformly on $\Theta$ in probability.
\end{lem}
\begin{proof}
For the purposes of this proof, fix $r\in(0,1)$ and define
\begin{equation}
\Pi_i^r
\coloneqq
\{j\in\Pi_i:T_{j-1}-t_i\ge r\},
\end{equation}
together with the truncated criteria
\begin{align}
\bar U_{n,d}^r(\theta)
&\coloneqq
\frac1n\sum_{i=1}^n\sum_{j\in\Pi_i^r}\delta_j
\left[
\log\frac{q_{ij}(\theta)}{\Delta_n}
+\frac{\bar X_{ij}^{2}}{q_{ij}(\theta)}
+\frac{q_{ij}(\theta)}4
\right],\\
U^r(\theta)
&\coloneqq
\int_r^1(1-x)
\left[
\log k_\theta(x)^2
+\frac{k_{\theta_0}(x)^2}{k_\theta(x)^2}
\right]dx.
\end{align}

In addition to the generic constants $C$ introduced above,
we use $C_r>0$ to denote generic constants in upper bounds
whose values may change from line to line.
Their only additional dependence is on $r$; in particular,
they are independent of $\theta$, the observation grids,
and the time and maturity variables.

By the triangle inequality,
\begin{align}
\sup_{\theta\in\Theta}|\bar U_{n,d}(\theta)-U(\theta)|
&\le
\sup_{\theta\in\Theta}
|\bar U_{n,d}(\theta)-\bar U_{n,d}^r(\theta)|+
\sup_{\theta\in\Theta}
|\bar U_{n,d}^r(\theta)-U^r(\theta)|\\
&\quad+
\sup_{\theta\in\Theta}|U^r(\theta)-U(\theta)|.
\label{eq:ulln-decomposition}
\end{align}
We first show that the middle term converges to zero in probability
for each fixed $r$.

The Gaussian moment formula and
Lemma~\ref{lemA1} yield, uniformly in
$\theta\in\Theta$ and $i$,
\begin{equation}
\left|
\sum_{j\in\Pi_i^r}
\frac{\delta_j}{q_{ij}(\theta)}
\bigl(\bar X_{ij}^2-E[\bar X_{ij}^2]\bigr)
\right|_2
\le
C\sum_{j\in\Pi_i}
\delta_j\frac{q_{ij}(\theta_0)}{\Delta_n}
\le C.
\end{equation}
The vectors $(\bar X_{ij})_{j\in\Pi_i^r}$ are independent across
the time index $i$. Thus, for every $\theta\in\Theta$,
\begin{equation}
\label{eq:ulln-centered-pointwise}
E\left|
\bar U_{n,d}^r(\theta)-E[\bar U_{n,d}^r(\theta)]
\right|^2
=\frac{1}{n^2}\sum_{i=1}^n E\left|\sum_{j\in\Pi_i^r}
\frac{\delta_j}{q_{ij}(\theta)}
\bigl(\bar X_{ij}^2-E[\bar X_{ij}^2]\bigr)\right|^2 \le \frac{C}{n}.
\end{equation}

To obtain uniform convergence, we use the uniform continuity of $k$
on $\Theta\times[r,1]$. For every $\varepsilon_k>0$, there exists
$\rho>0$ such that
$\sup_{x\in[r,1]}|k_\theta(x)-k_\vartheta(x)|\le\varepsilon_k$
whenever $\theta,\vartheta\in\Theta$ and
$|\theta-\vartheta|<\rho$.
On the truncated region, the kernels are uniformly bounded and
$q_{ij}(\theta)\ge c\Delta_n$. It follows that
$|q_{ij}(\theta)-q_{ij}(\vartheta)|
\le C_r\varepsilon_k\Delta_n$, which also controls the differences
of the logarithms and reciprocals appearing in the criterion.
Hence, using the Gaussian moment formula and Lemma~\ref{lemA1},
\begin{equation}
\label{eq:ulln-modulus}
E\left[
\sup_{\substack{\theta,\vartheta\in\Theta\\
|\theta-\vartheta|<\rho}}
|\bar U_{n,d}^r(\theta)-\bar U_{n,d}^r(\vartheta)|
\right]\le
C_r\varepsilon_k
E\left[
1+\frac1n\sum_{i=1}^n\sum_{j\in\Pi_i^r}
\delta_j\frac{\bar X_{ij}^2}{\Delta_n}
\right]
\le C_r\varepsilon_k.
\end{equation}
By compactness, $\Theta$ admits a finite $\rho$-net
$\theta^{(1)},\ldots,\theta^{(N)}$, where $N$ depends only on
$r$ and $\varepsilon_k$.
Approximating each parameter by a point of this net and applying
Markov's inequalities to
\eqref{eq:ulln-centered-pointwise} and \eqref{eq:ulln-modulus},
respectively, gives, for every $\varepsilon>0$,
\begin{equation}
P\left(
\sup_{\theta\in\Theta}
|\bar U_{n,d}^r(\theta)-E[\bar U_{n,d}^r(\theta)]|>\varepsilon
\right)
\le
\frac{CN}{n\varepsilon^2}+\frac{C_r\varepsilon_k}{\varepsilon}.
\end{equation}
For each fixed $r\in(0,1)$, letting $n,d\to\infty$ first
and then $\varepsilon_k\downarrow0$ yields
\begin{equation}
\label{eq:ulln-centered-uniform}
\sup_{\theta\in\Theta}
|\bar U_{n,d}^r(\theta)-E[\bar U_{n,d}^r(\theta)]|
\xrightarrow{p}0.
\end{equation}

We next identify the limit of the expectations. For $h>0$, define
\begin{equation}
\omega^r(h)
\coloneqq
\sup_{\substack{\theta\in\Theta,\ x,y\in[r,1]\\|x-y|\le h}}
|k_\theta(x)-k_\theta(y)|.
\end{equation}
Assumption~\ref{ass:K0} and compactness imply that
$\omega^r(h)\to0$ as $h\downarrow0$.

For $j\in\Pi_i^r$, all maturity lags in the defining integrals
belong to $[r,1]$ and differ from $T_j-t_i$ by at most
$\Delta_n+\Delta_d$. Thus
$|\bar k_\theta^j(s)-k_\theta(T_j-t_i)|
\le\omega^r(\Delta_n+\Delta_d)$ for
$s\in[t_{i-1},t_i]$.
Since the kernels are uniformly bounded on the truncated region,
squaring and averaging over the time interval gives
\begin{equation}
\label{eq:ulln-q-approximation}
\sup_{\substack{\theta\in\Theta,\ 1\le i\le n\\j\in\Pi_i^r}}
\left|
\frac{q_{ij}(\theta)}{\Delta_n}
-k_\theta(T_j-t_i)^2
\right|
\le
C_r\omega^r(\Delta_n+\Delta_d).
\end{equation}
Both quantities in \eqref{eq:ulln-q-approximation} are uniformly
bounded above and bounded away from zero. Consequently, the errors
in the corresponding logarithms and ratios are also bounded by
$C_r\omega^r(\Delta_n+\Delta_d)$. Moreover,
\begin{equation}
\sup_{\substack{\theta\in\Theta,\ 1\le i\le n\\j\in\Pi_i^r}}
\left[
\frac{q_{ij}(\theta_0)^2}{q_{ij}(\theta)}
+q_{ij}(\theta)
\right]
\le C_r\Delta_n.
\end{equation}
The Gaussian moment formula and
$\sum_i\sum_{j\in\Pi_i^r}\delta_j/n\le1$ yield
\begin{align}
\label{eq:ulln-mean-riemann}
&\sup_{\theta\in\Theta}
\left|
E[\bar U_{n,d}^r(\theta)]
-
\frac1n\sum_{i=1}^n\sum_{j\in\Pi_i^r}\delta_j
\left[
\log k_\theta(T_j-t_i)^2
+\frac{k_{\theta_0}(T_j-t_i)^2}{k_\theta(T_j-t_i)^2}
\right]
\right|\\
&\qquad\le
C_r\left[\omega^r(\Delta_n+\Delta_d)+\Delta_n\right].
\end{align}

The sum in \eqref{eq:ulln-mean-riemann} is a Riemann sum over
the rectangles $(t_{i-1},t_i]\times(T_{j-1},T_j]$ with
$j\in\Pi_i^r$, whose areas are $\Delta_n\delta_j$.
These rectangles lie in
$\{(t,u)\in[0,1]^2\mid u-t\ge r\}$.
Within this region, their union leaves uncovered only points in
$\{t=0\}$, $\{u>T_d\}$, or
$\{r\le u-t<r+\Delta_n+\Delta_d\}$.
Indeed, if $t>0$, $u\le T_d$, and
$u-t\ge r+\Delta_n+\Delta_d$, the rectangle containing $(t,u)$
satisfies
$T_{j-1}-t_i\ge u-t-\delta_j-\Delta_n\ge r$.
The uncovered area is therefore at most
$(1-T_d)+\Delta_n+\Delta_d$.

The integrand is uniformly bounded on the truncated region, and
its variation within each rectangle is at most
$C_r\omega^r(\Delta_n+\Delta_d)$.
Consequently, the Riemann sum converges uniformly in $\theta$ to $U^r$.
Combining the rectangle approximation and the uncovered-area bound
with \eqref{eq:ulln-mean-riemann} yields
\begin{equation}
\sup_{\theta\in\Theta}
|E[\bar U_{n,d}^r(\theta)]-U^r(\theta)|
\le
C_r\left[
\omega^r(\Delta_n+\Delta_d)
+\Delta_n+\Delta_d+(1-T_d)
\right].
\end{equation}
Together with \eqref{eq:ulln-centered-uniform}, this proves
\begin{equation}
\label{eq:ulln-truncated}
\sup_{\theta\in\Theta}
|\bar U_{n,d}^r(\theta)-U^r(\theta)|
\xrightarrow{p}0
\end{equation}
for each fixed $r$.

We now control the first term in \eqref{eq:ulln-decomposition}.
The Gaussian moment formula and
Lemma~\ref{lemA1} give
\begin{equation}
E\left[
\sup_{\theta\in\Theta}
\left|
\log\frac{q_{ij}(\theta)}{\Delta_n}
+\frac{\bar X_{ij}^2}{q_{ij}(\theta)}
+\frac{q_{ij}(\theta)}4
\right|
\right]
\le
C\left[
1+|\log(T_j-t_{i-1})|
+(T_j-t_{i-1})^{2\alpha-1}
\right].
\end{equation}
For $j\in\Pi_i\setminus\Pi_i^r$, we have
$0<T_{j-1}-t_i<r$ and $T_j-t_i<r+\Delta_d$.
Furthermore, $T_j-t_{i-1}\ge u-t_i>0$ for
$u\in[T_{j-1},T_j]$.
Since $1+|\log x|$ and $x^{2\alpha-1}$ are nonincreasing on
$(0,1]$, comparison with integrals over the disjoint maturity
intervals gives
\begin{align}
\label{eq:ulln-boundary-random}
E\left[
\sup_{\theta\in\Theta}
|\bar U_{n,d}(\theta)-\bar U_{n,d}^r(\theta)|
\right]&\le
C\int_0^{\min\{r+\Delta_d,1\}}
\left[1+|\log x|+x^{2\alpha-1}\right]dx\\
&\le
C\left[
(r+\Delta_d)^{2\alpha}
+(r+\Delta_d)\{1+|\log(r+\Delta_d)|\}
\right].
\end{align}

Finally, Lemma~\ref{lemA1} gives
$|\log k_\theta(x)^2|\le C(1+|\log x|)$ and
$k_{\theta_0}(x)^2/k_\theta(x)^2\le Cx^{2\alpha-1}$,
uniformly in $\theta$. Hence
\begin{align}
\label{eq:ulln-boundary-deterministic}
\sup_{\theta\in\Theta}|U(\theta)-U^r(\theta)|
&\le
C\int_0^r
\left[1+|\log x|+x^{2\alpha-1}\right]dx\\
&\le
C\left[r^{2\alpha}+r(1+|\log r|)\right].
\end{align}
Since $\Delta_d\to0$, \eqref{eq:ulln-boundary-random} gives
\begin{equation}
\limsup_{n,d\to\infty}
E\left[
\sup_{\theta\in\Theta}
|\bar U_{n,d}(\theta)-\bar U_{n,d}^r(\theta)|
\right]
\le
C\left[r^{2\alpha}+r(1+|\log r|)\right].
\end{equation}
This upper bound and the deterministic bound in
\eqref{eq:ulln-boundary-deterministic} both tend to zero as
$r\downarrow0$.

For each fixed $r$, \eqref{eq:ulln-truncated} shows that the
middle term in \eqref{eq:ulln-decomposition} converges to zero
in probability as $n,d\to\infty$.
Combining this convergence with the boundary estimates and
applying Markov's inequality, we obtain the asserted uniform
convergence in probability by first letting $n,d\to\infty$
for fixed $r$ and then letting $r\downarrow0$.
\end{proof}

\begin{proof}[Proof of Theorem~\ref{thm:uniform-contrast}]
By the triangle inequality,
\begin{equation*}
    \sup_{\theta\in\Theta}|U_{n,d}(\theta)-U(\theta)|
    \le
    \sup_{\theta\in\Theta}
    |U_{n,d}(\theta)-\bar U_{n,d}(\theta)|
    +
    \sup_{\theta\in\Theta}
    |\bar U_{n,d}(\theta)-U(\theta)|.
\end{equation*}
Both terms on the right converge to zero in probability
by the two preceding lemmas.
\end{proof}
\section{Proof of asymptotic normality}
\label{A3}
Throughout this appendix, we assume
Assumptions~\ref{ass:K0}--\ref{ass:K5} and
Assumption~\ref{ass:V0}, with
$\theta_0\in\operatorname{Int}(\Theta)$.

As in the preceding appendix, the letters $c$ and $C$
denote positive constants that are independent of
$\theta\in\Theta$, $n$,$d$, and the observation grids.
Their values may change from one occurrence to the next.
We continue to use the Gaussian proxy contrast
$\bar U_{n,d}$ defined in~\eqref{barU}.
The matrices $B$, $D$, and $\Gamma$ are defined
in~\eqref{eq:limiting-hessian},
\eqref{eq:limiting-score-covariance}, and
\eqref{eq:asymptotic-covariance}, respectively.
To separate the diagonal boundary from the interior,
for each $r\in(0,1)$, define
\begin{equation}
    \Pi_i^r
    \coloneqq
    \{j\in\Pi_i \mid T_{j-1}-t_i\ge r\}.
\end{equation}
Define the truncated contrasts by
\begin{equation}
      \bar U_{n,d}^r(\theta)
    \coloneqq
    \frac1n\sum_{i=1}^n\sum_{j\in\Pi_i^r}\delta_j
    \left[
        \log\frac{q_{ij}(\theta)}{\Delta_n}
        +\frac{\bar X_{ij}^2}{q_{ij}(\theta)}
        +\frac{q_{ij}(\theta)}4
    \right].  
\end{equation}
In estimates on the truncated region, $C_r$ denotes a
positive constant that may additionally depend on $r$,
but remains independent of the parameter, the observation
grids, and the time and maturity variables.
Its value may change from one occurrence to the next.

Let
\begin{equation}
        \Omega_{n,d}
    \coloneqq
    \bigl\{
        \hat\theta_{n,d}\in\operatorname{Int}(\Theta)
    \bigr\}.
\end{equation}
By Corollary~\ref{cor:consistency} and the interiority of
$\theta_0$, the probability of $\Omega_{n,d}$ tends to one.

Define
\begin{equation}
    B_{n,d}
    \coloneqq
    \int_0^1
    \nabla_\theta^2 U_{n,d}
    \bigl(
        \theta_0+s(\hat\theta_{n,d}-\theta_0)
    \bigr)\,ds.
\end{equation}
This integral is well-defined by
Lemma~\ref{lemA2} and the convexity of $\Theta$.

On $\Omega_{n,d}$, the first-order condition gives
$\nabla_\theta U_{n,d}(\hat\theta_{n,d})=0$.
Hence the integral form of Taylor's formula yields,
on $\Omega_{n,d}$,
\begin{equation}
    0
    =
    \nabla_\theta U_{n,d}(\theta_0)
    +B_{n,d}(\hat\theta_{n,d}-\theta_0).
    \label{eq:estimator-taylor-expansion}
\end{equation}

\begin{lem}
\label{lem:score-comparison}
For every $\varepsilon\in(0,\alpha\wedge\beta)$,
there exists a constant $C$ such that
\begin{equation}
    E\!\left[
        \left|
            \nabla_\theta U_{n,d}(\theta_0)
            -
            \nabla_\theta\bar U_{n,d}(\theta_0)
        \right|
    \right]
    \le
    C
    \Delta_d^{1+\alpha\wedge\beta-\varepsilon}.
    \label{eq:score-comparison-bound}
\end{equation}
Consequently, if
\eqref{eq:maximal-mesh-condition} holds for some
$\varepsilon\in(0,\alpha\wedge\beta)$, then
\begin{equation}
        \sqrt n\,
    \bigl(
        \nabla_\theta U_{n,d}(\theta_0)
        -
        \nabla_\theta\bar U_{n,d}(\theta_0)
    \bigr)
\end{equation}
converges to zero in probability.
\end{lem}
\begin{proof}
Differentiating the contrast gives
\begin{equation}
\label{eq:contrast-gradient-formula}
    \nabla_\theta U_{n,d}(\theta)
    =
    \frac1n\sum_{i=1}^n\sum_{j\in\Pi_i}
    \delta_j\nabla_\theta\log q_{ij}(\theta)
    \left[
        1-\frac{X_{ij}^2}{q_{ij}(\theta)}
        +\frac{q_{ij}(\theta)}4
    \right].
\end{equation}
Replacing $X_{ij}$ by $\bar X_{ij}$ gives the corresponding
formula for $\nabla_\theta\bar U_{n,d}(\theta)$.
Subtracting these expressions at $\theta_0$, we obtain
\begin{equation}
    \nabla_\theta U_{n,d}(\theta_0)
    -\nabla_\theta\bar U_{n,d}(\theta_0)
    =
    -\frac1n\sum_{i=1}^n\sum_{j\in\Pi_i}
    \delta_j\nabla_\theta\log q_{ij}(\theta_0)
    \frac{X_{ij}^2-\bar X_{ij}^2}{q_{ij}(\theta_0)}.
\end{equation}
Fix $\varepsilon\in(0,\alpha\wedge\beta)$.
Applying Lemma~\ref{lemA2} with exponent $\varepsilon/2$
and Lemma~\ref{lemA5} yields
\begin{align}
\label{eq:score-comparison-reduction}
    E\left[
        \left|
            \nabla_\theta U_{n,d}(\theta_0)
            -\nabla_\theta\bar U_{n,d}(\theta_0)
        \right|
    \right]&\le
    \frac Cn\sum_{i=1}^n\sum_{j\in\Pi_i}
    \frac{\delta_j^{2+\alpha\wedge\beta}}
         {(T_j-t_i)(T_j-t_{i-1})^{\varepsilon/2}}\\
    &\le
    \frac{C\Delta_d^{1+\alpha\wedge\beta-\varepsilon}}{n}
    \sum_{i=1}^n\sum_{j\in\Pi_i}
    \delta_j(T_j-t_i)^{-1+\varepsilon/2},
\end{align}
where we used
$\delta_j<T_j-t_i\le T_j-t_{i-1}$ and
$\delta_j\le\Delta_d$ for every $j \in \Pi_i$.

Since $\varepsilon<\alpha\le1/2$, comparison over the
disjoint maturity intervals gives
\begin{align}
    \sum_{j\in\Pi_i}
    \delta_j(T_j-t_i)^{-1+\varepsilon/2}
    &\le
    \sum_{j\in\Pi_i}
    \int_{T_{j-1}}^{T_j}(u-t_i)^{-1+\varepsilon/2}du\\
    &\le
    \int_{t_i}^1(u-t_i)^{-1+\varepsilon/2}du
    =
    \frac2\varepsilon(1-t_i)^{\varepsilon/2}
    \le\frac2\varepsilon.
\end{align}
Substituting this bound into
\eqref{eq:score-comparison-reduction} proves
\eqref{eq:score-comparison-bound}.
\end{proof}

\begin{lem}
\label{lemC1}
    There exists a constant $C$ such that
    \begin{equation}
        E\left[\sup_{\theta \in \Theta} \left|\nabla^2_\theta U_{n,d}(\theta) - \nabla^2_\theta \bar U_{n,d}(\theta)\right|\right] \leq C \Delta_d^{\alpha \wedge \beta}.
    \end{equation}
    Moreover, for any $r \in (0,1)$,
    \begin{equation}
             E\left[\sup_{\theta \in \Theta} \left|\nabla^2_\theta \bar U_{n,d}(\theta) - \nabla^2_\theta \bar U_{n,d}^r(\theta)\right|\right] \leq C (r + \Delta_d)^\alpha.
    \end{equation}
\end{lem}
\begin{proof}
Fix $i$ and $j\in\Pi_i$.
The chain rule gives
\begin{equation}
    \nabla_\theta^2\bigl(q_{ij}^{-1}\bigr)(\theta)
    =
    \frac{1}{q_{ij}(\theta)}
    \left[
        \nabla_\theta\log q_{ij}(\theta)
        \nabla_\theta\log q_{ij}(\theta)^\top
        -
        \nabla_\theta^2\log q_{ij}(\theta)
    \right].
\end{equation}
Applying Lemma~\ref{lemA2} with exponents $\alpha/2$ and
$\alpha$, and then using Lemma~\ref{lemA1}, we obtain
\begin{equation}
\label{eq:hessian-weighted-reciprocal-bound}
    q_{ij}(\theta_0)
    \sup_{\theta\in\Theta}
    \left|
        \nabla_\theta^2\bigl(q_{ij}^{-1}\bigr)(\theta)
    \right|
    \le
    C\frac{q_{ij}(\theta_0)}{\Delta_n}
    (T_j-t_{i-1})^{-\alpha}
    \le C(T_j-t_{i-1})^{\alpha-1}.
\end{equation}
Similarly, the identity
\begin{equation}
    \nabla_\theta^2 q_{ij}(\theta)
    =
    q_{ij}(\theta)
    \left[
        \nabla_\theta^2\log q_{ij}(\theta)
        +
        \nabla_\theta\log q_{ij}(\theta)
        \nabla_\theta\log q_{ij}(\theta)^\top
    \right],
\end{equation}
together with $\sup_{\theta\in\Theta}q_{ij}(\theta)\le C$
and Lemma~\ref{lemA2}, yields
\begin{equation}
    \sup_{\theta\in\Theta}
    \left[
        |\nabla_\theta^2\log q_{ij}(\theta)|
        +
        |\nabla_\theta^2 q_{ij}(\theta)|
    \right]
    \le C(T_j-t_{i-1})^{-\alpha}.
\end{equation}
Since $0<T_j-t_{i-1}\le1$ and $\alpha\le1/2$, we have
$(T_j-t_{i-1})^{-\alpha}
\le (T_j-t_{i-1})^{\alpha-1}$.
Combining this observation with
\eqref{eq:hessian-weighted-reciprocal-bound}, we obtain
\begin{equation}
\label{eq:hessian-regular-coefficient-bound}
    \sup_{\theta\in\Theta}
    \left[
        |\nabla_\theta^2\log q_{ij}(\theta)|
        +
        |\nabla_\theta^2 q_{ij}(\theta)|
    \right]
    +
    q_{ij}(\theta_0)
    \sup_{\theta\in\Theta}
    \left|
        \nabla_\theta^2\bigl(q_{ij}^{-1}\bigr)(\theta)
    \right|
    \le C(T_j-t_{i-1})^{\alpha-1}.
\end{equation}

Since $T_j-t_i>\delta_j$, Lemma~\ref{lemA5} implies
\begin{equation}
\label{eq:hessian-squared-increment-error}
    E|X_{ij}^2-\bar X_{ij}^2|
    \le
    Cq_{ij}(\theta_0)
    \frac{\delta_j^{1+\alpha\wedge\beta}}{T_j-t_i}
    \le
    C\Delta_d^{\alpha\wedge\beta}q_{ij}(\theta_0), \quad E[\bar X_{ij}^2]
    \le Cq_{ij}(\theta_0).
\end{equation}

Differentiating the contrast gives
\begin{equation}
\label{eq:contrast-hessian-formula}
    \nabla_\theta^2 U_{n,d}(\theta)
    =
    \frac1n\sum_{i=1}^n\sum_{j\in\Pi_i}\delta_j
    \left[
        \nabla_\theta^2\log q_{ij}(\theta)
        +X_{ij}^2
        \nabla_\theta^2\bigl(q_{ij}^{-1}\bigr)(\theta)
        +\frac14\nabla_\theta^2q_{ij}(\theta)
    \right].
\end{equation}
Replacing $X_{ij}$ by $\bar X_{ij}$ gives the corresponding
formula for $\nabla_\theta^2\bar U_{n,d}(\theta)$.
Restricting the inner sum further to $j\in\Pi_i^r$ gives
the formula for $\nabla_\theta^2\bar U_{n,d}^r(\theta)$.

By the definitions of the criteria,
\begin{equation}
    \nabla_\theta^2 U_{n,d}(\theta)
    -
    \nabla_\theta^2 \bar U_{n,d}(\theta)
    =
    \frac1n\sum_{i=1}^n\sum_{j\in\Pi_i}
    \delta_j(X_{ij}^2-\bar X_{ij}^2)
    \nabla_\theta^2\bigl(q_{ij}^{-1}\bigr)(\theta).
\end{equation}
Combining these identities with
\eqref{eq:hessian-regular-coefficient-bound} and
\eqref{eq:hessian-squared-increment-error}, we obtain
\begin{align}
\label{eq:hessian-transfer-sum}
    &E\left[
        \sup_{\theta\in\Theta}
        \left|
            \nabla_\theta^2 U_{n,d}(\theta)
            -
            \nabla_\theta^2 \bar U_{n,d}(\theta)
        \right|
    \right]
    \le
    \frac{C\Delta_d^{\alpha\wedge\beta}}{n}
    \sum_{i=1}^n\sum_{j\in\Pi_i}
    \delta_j(T_j-t_{i-1})^{\alpha-1},\\
\label{eq:hessian-truncation-sum}
    &E\left[
        \sup_{\theta\in\Theta}
        \left|
            \nabla_\theta^2\bar U_{n,d}(\theta)
            -
            \nabla_\theta^2\bar U_{n,d}^r(\theta)
        \right|
    \right]
    \le
    \frac Cn\sum_{i=1}^n
    \sum_{j\in\Pi_i\setminus\Pi_i^r}
    \delta_j(T_j-t_{i-1})^{\alpha-1}.
\end{align}

It remains to estimate the maturity sums.
For $j\in\Pi_i$ and $u\in[T_{j-1},T_j]$,
we have $T_j-t_{i-1}\ge u-t_i>0$.
Thus, by disjointness of the maturity intervals,
\begin{equation}
\label{eq:hessian-full-maturity-sum}
    \sum_{j\in\Pi_i}
    \delta_j(T_j-t_{i-1})^{\alpha-1}
    \le
    \int_{t_i}^1(u-t_i)^{\alpha-1}\,du
    \le \frac1\alpha.
\end{equation}
Substituting \eqref{eq:hessian-full-maturity-sum} into
\eqref{eq:hessian-transfer-sum} proves the first assertion.

If $j\in\Pi_i\setminus\Pi_i^r$, then
$T_j-t_i<r+\Delta_d$, so the same comparison gives
\begin{equation}
\label{ESOFdnsovkdWM}
    \sum_{j\in\Pi_i\setminus\Pi_i^r}
    \delta_j(T_j-t_{i-1})^{\alpha-1}
    \le
    \int_0^{\min\{r+\Delta_d,1\}}x^{\alpha-1}\,dx
    \le \frac{(r+\Delta_d)^\alpha}{\alpha}.
\end{equation}
Substituting this bound into
\eqref{eq:hessian-truncation-sum} proves the second assertion.
\end{proof}

For $j,\ell\in\Pi_i$, define
\begin{equation}
\label{eq:cross-maturity-q}
    q_i^{j\ell}(\theta)
    \coloneqq
    \int_{t_{i-1}}^{t_i}
    \bar k_\theta^j(s)\bar k_\theta^\ell(s)ds.
\end{equation}
In particular, $q_i^{jj}(\theta)=q_{ij}(\theta)$.

For $\theta\in\Theta$ and $t\in[0,1]$, define the
deterministic matrices
\begin{equation}
\label{eq:discrete-hessian-matrix}
    \mathcal B_{n,d}(\theta)
    \coloneqq
    \frac1n\sum_{i=1}^n\sum_{j\in\Pi_i}\delta_j
    \left(1+\frac12q_{ij}(\theta)\right)
    \nabla_\theta\log q_{ij}(\theta)
    \nabla_\theta\log q_{ij}(\theta)^\top
\end{equation}
and
\begin{equation}
\label{eq:discrete-score-covariance}
\begin{aligned}
    \mathcal D_{n,d}(t,\theta)
    \coloneqq{}&
    \frac1n\sum_{i=1}^{\lfloor nt\rfloor}
    \sum_{j,\ell\in\Pi_i}\delta_j\delta_\ell
    \left[
        \frac{2q_i^{j\ell}(\theta)^2}
             {q_{ij}(\theta)q_{i\ell}(\theta)}
        +q_i^{j\ell}(\theta)
    \right]
    \nabla_\theta\log q_{ij}(\theta)
    \nabla_\theta\log q_{i\ell}(\theta)^\top.
\end{aligned}
\end{equation}
Their limiting counterparts are
\begin{equation}
\label{eq:parameterized-hessian-limit}
    B(\theta)
    \coloneqq
    4\int_0^1(1-x)
    \nabla_\theta\log k_\theta(x)
    \nabla_\theta\log k_\theta(x)^\top dx
\end{equation}
and
\begin{equation}
\label{eq:parameterized-score-covariance-limit}
    D(t,\theta)
    \coloneqq
    8\int_0^t
    \left(
        \int_0^{1-s}\nabla_\theta\log k_\theta(x)dx
    \right)
    \left(
        \int_0^{1-s}\nabla_\theta\log k_\theta(x)dx
    \right)^\top ds.
\end{equation}
By \eqref{eq:limiting-hessian} and
\eqref{eq:limiting-score-covariance},
\begin{equation}
    B(\theta_0)=B,
    \qquad
    D(1,\theta_0)=D.
\end{equation}

\begin{lem}
\label{lem:deterministic-matrix-limits}
As $n,d\to\infty$,
\begin{equation}
\label{eq:uniform-hessian-matrix-limit}
    \sup_{\theta\in\Theta}
    |\mathcal B_{n,d}(\theta)-B(\theta)|
    \longrightarrow0
\end{equation}
and, for every $t\in[0,1]$,
\begin{equation}
\label{eq:uniform-score-covariance-limit}
    \sup_{\theta\in\Theta}
    |\mathcal D_{n,d}(t,\theta)-D(t,\theta)|
    \longrightarrow0.
\end{equation}
\end{lem}
\begin{proof}
Fix $t\in[0,1]$.
For $r\in(0,1)$, let $\mathcal B_{n,d}^r$ and
$\mathcal D_{n,d}^r$ be obtained from
\eqref{eq:discrete-hessian-matrix} and \eqref{eq:discrete-score-covariance}, respectively,
by replacing each maturity index set $\Pi_i$ with $\Pi_i^r$.
Define their limiting counterparts by
\begin{align}
    B^r(\theta)
    &\coloneqq
    4\int_r^1(1-x)
    \nabla_\theta\log k_\theta(x)
    \nabla_\theta\log k_\theta(x)^\top dx,\\
    D^r(t,\theta)
    &\coloneqq
    8\int_0^t
    \left(
        \int_0^{1-s}
        \mathbf 1_{\{x\ge r\}}
        \nabla_\theta\log k_\theta(x)dx
    \right)
    \left(
        \int_0^{1-s}
        \mathbf 1_{\{x\ge r\}}
        \nabla_\theta\log k_\theta(x)dx
    \right)^\top ds.
\end{align}
We first control the truncation errors and then establish
uniform convergence on the truncated region.

By the Cauchy--Schwarz inequality and Lemma~\ref{lemA1},
\begin{equation}
    q_i^{j\ell}(\theta)
    \le
    \sqrt{q_{ij}(\theta)q_{i\ell}(\theta)}
    \le C.
\end{equation}
Applying Lemma~\ref{lemA2} with exponent $\alpha/2$
and using $\alpha\le1/2$, we obtain
\begin{align}
\label{eq:discrete-matrix-truncation}
    \sup_{\theta\in\Theta}
    |\mathcal B_{n,d}(\theta)-\mathcal B_{n,d}^r(\theta)|
    &\le
    \frac Cn\sum_{i=1}^n
    \sum_{j\in\Pi_i\setminus\Pi_i^r}
    \delta_j(T_j-t_{i-1})^{\alpha-1}\\
    &\le C(r+\Delta_d)^\alpha,
\end{align}
and
\begin{align}
\label{eq:discrete-covariance-truncation}
    &\sup_{\theta\in\Theta}
    |\mathcal D_{n,d}(t,\theta)
      -\mathcal D_{n,d}^r(t,\theta)|\\
    &\qquad\le
    \frac Cn\sum_{i=1}^{\lfloor nt\rfloor}
    \left(
        \sum_{j\in\Pi_i\setminus\Pi_i^r}
        \delta_j(T_j-t_{i-1})^{\alpha-1}
    \right)
    \left(
        \sum_{\ell\in\Pi_i}
        \delta_\ell(T_\ell-t_{i-1})^{\alpha-1}
    \right)\\
    &\qquad\le C(r+\Delta_d)^\alpha.
\end{align}
For the covariance bound, each omitted pair has at least
one index in $\Pi_i\setminus\Pi_i^r$.
The final inequalities follow from
\eqref{eq:hessian-full-maturity-sum} and
\eqref{ESOFdnsovkdWM}.

For the limiting matrices,
Assumption~\ref{ass:K4} gives
\begin{equation}
    |B(\theta)-B^r(\theta)|
    \le C\int_0^r x^{\alpha-1}dx
    \le Cr^\alpha,
\end{equation}
uniformly in $\theta\in\Theta$.
Expanding the difference of the outer products and applying
the triangle inequality, we obtain
\begin{align}
    |D(t,\theta)-D^r(t,\theta)|
    &\le
    16\int_0^t
    \left(
        \int_0^{r\wedge(1-s)}
        |\nabla_\theta\log k_\theta(x)|dx
    \right)
    \left(
        \int_0^{1-s}
        |\nabla_\theta\log k_\theta(y)|dy
    \right)ds\\
    &\le Cr^\alpha,
\end{align}
uniformly in $\theta\in\Theta$, where the last inequality
follows from Assumption~\ref{ass:K4}.
Hence
\begin{equation}
\label{eq:limiting-matrix-truncation}
    \sup_{\theta\in\Theta}|B(\theta)-B^r(\theta)|
    +
    \sup_{\theta\in\Theta}|D(t,\theta)-D^r(t,\theta)|
    \le Cr^\alpha.
\end{equation}

We now fix $r\in(0,1)$.
For $h>0$, define
\begin{equation}
    \omega^r(h)
    \coloneqq
    \sup_{\substack{\theta\in\Theta,\ x,y\in[r,1]\\
                    |x-y|\le h}}
    \left[
        |k_\theta(x)-k_\theta(y)|
        +|\nabla_\theta k_\theta(x)
          -\nabla_\theta k_\theta(y)|
    \right].
\end{equation}
Assumptions~\ref{ass:K0} and \ref{ass:K3}, together with
compactness, imply that $\omega^r(h)\to0$ as $h\downarrow0$.

For $j,\ell\in\Pi_i^r$, all lags in the defining integrals
belong to $[r,1]$ and differ from their corresponding
values $T_j-t_i$ and $T_\ell-t_i$ by at most
$\Delta_n+\Delta_d$.
Applying the modulus of continuity $\omega^r$ to the
defining integrals for $q_{ij}$ and $q_i^{j\ell}$
and to \eqref{WSDFOISnfsdgoiNAOisfgndfgi},
and using $c\le q_{ij}(\theta)/\Delta_n\le C_r$,
we obtain
\begin{align}
\label{eq:matrix-riemann-approximation}
    &\sup_{\theta\in\Theta}
    \left|
        \mathcal B_{n,d}^r(\theta)
        -
        \frac4n\sum_{i=1}^n\sum_{j\in\Pi_i^r}
        \delta_j
        \nabla_\theta\log k_\theta(T_j-t_i)
        \nabla_\theta\log k_\theta(T_j-t_i)^\top
    \right|\\
    &\quad+
    \sup_{\theta\in\Theta}
    \left|
        \mathcal D_{n,d}^r(t,\theta)
        -
        \frac8n\sum_{i=1}^{\lfloor nt\rfloor}
        \sum_{j,\ell\in\Pi_i^r}
        \delta_j\delta_\ell
        \nabla_\theta\log k_\theta(T_j-t_i)
        \nabla_\theta\log k_\theta(T_\ell-t_i)^\top
    \right|\\
    &\qquad\le
    C_r\left[
        \omega^r(\Delta_n+\Delta_d)+\Delta_n
    \right].
\end{align}
Here $\nabla_\theta\log q_{ij}(\theta)$ and
$q_i^{j\ell}(\theta)^2/
\{q_{ij}(\theta)q_{i\ell}(\theta)\}$
are approximated by
$2\nabla_\theta\log k_\theta(T_j-t_i)$ and $1$,
respectively, each with error at most
$C_r\omega^r(\Delta_n+\Delta_d)$.
The two logarithmic derivatives give the factor $4$;
the additional coefficient $2$ in the covariance matrix
gives the factor $8$.
Moreover,
$q_{ij}(\theta)+q_i^{j\ell}(\theta)\le C_r\Delta_n$
on the truncated region.
Thus the terms involving $q_{ij}(\theta)/2$ and $q_i^{j\ell}(\theta)$
contribute at most $C_r\Delta_n$.

The Riemann-sum argument following
\eqref{eq:ulln-mean-riemann} applies to both sums,
using rectangles for the first and boxes for the second.
On the truncated domains, the integrands are uniformly
bounded in $\theta$, with variation within each cell
at most $C_r\omega^r(\Delta_n+\Delta_d)$.
The uncovered area or volume is bounded by
$C[\Delta_n+\Delta_d+(1-T_d)]$.
For the second sum, this includes the time strip
$(\lfloor nt\rfloor/n,t]$, whose width is at most $\Delta_n$.
Changing variables from maturities to lags and applying
Fubini's theorem identifies the limiting integrals as
$B^r(\theta)$ and $D^r(t,\theta)$, respectively.
Consequently,
\begin{align}
\label{eq:truncated-matrix-convergence}
    &\sup_{\theta\in\Theta}
    |\mathcal B_{n,d}^r(\theta)-B^r(\theta)|
    +
    \sup_{\theta\in\Theta}
    |\mathcal D_{n,d}^r(t,\theta)-D^r(t,\theta)|\\
    &\qquad\le
    C_r\left[
        \omega^r(\Delta_n+\Delta_d)
        +\Delta_n+\Delta_d+(1-T_d)
    \right]
    \longrightarrow0.
\end{align}

Combining
\eqref{eq:discrete-matrix-truncation},
\eqref{eq:discrete-covariance-truncation},
\eqref{eq:limiting-matrix-truncation}, and
\eqref{eq:truncated-matrix-convergence}
with the triangle inequality gives
\begin{equation}
    \limsup_{n,d\to\infty}
    \left[
        \sup_{\theta\in\Theta}
        |\mathcal B_{n,d}(\theta)-B(\theta)|
        +
        \sup_{\theta\in\Theta}
        |\mathcal D_{n,d}(t,\theta)-D(t,\theta)|
    \right]
    \le Cr^\alpha.
\end{equation}
Letting $r\downarrow0$ proves both assertions.
\end{proof}

\begin{lem}
\label{lem:integrated-hessian}
The matrices $B_{n,d}$ 
converge in probability to $B$.
\end{lem}
\begin{proof}
Define
\begin{equation}
    \bar B_{n,d}
    \coloneqq
    \int_0^1
    \nabla_\theta^2\bar U_{n,d}
    \bigl(
        \theta_0+s(\hat\theta_{n,d}-\theta_0)
    \bigr)ds.
\end{equation}
By the triangle inequality,
\begin{equation}
\label{eq:integrated-hessian-decomposition}
    |B_{n,d}-B|
    \le
    |B_{n,d}-\bar B_{n,d}|+
    |\bar B_{n,d}
      -\nabla_\theta^2\bar U_{n,d}(\theta_0)|
    +
    |\nabla_\theta^2\bar U_{n,d}(\theta_0)-B|.
\end{equation}
Convexity of $\Theta$ and Lemma~\ref{lemC1} give
\begin{equation}
\label{eq:integrated-hessian-transfer}
    E|B_{n,d}-\bar B_{n,d}|
    \le
    E\left[
        \sup_{\theta\in\Theta}
        |\nabla_\theta^2U_{n,d}(\theta)
          -\nabla_\theta^2\bar U_{n,d}(\theta)|
    \right]
    \le C\Delta_d^{\alpha\wedge\beta}.
\end{equation}

We next establish convergence of the proxy Hessian at
$\theta_0$.
Since
$E[\bar X_{ij}^2]
=q_{ij}(\theta_0)+q_{ij}(\theta_0)^2/4$,
\eqref{eq:contrast-hessian-formula} and the logarithmic
derivative identities give\\
$E[\nabla_\theta^2\bar U_{n,d}(\theta_0)]=\mathcal B_{n,d}(\theta_0)$.
Centering the same representation yields
\begin{equation}
    \nabla_\theta^2\bar U_{n,d}(\theta_0)
    -\mathcal B_{n,d}(\theta_0)
    =
    \frac1n\sum_{i=1}^n\sum_{j\in\Pi_i}\delta_j
    \bigl(\bar X_{ij}^2-E[\bar X_{ij}^2]\bigr)
    \nabla_\theta^2\bigl(q_{ij}^{-1}\bigr)(\theta_0).
\end{equation}
The Gaussian moment formula and
\eqref{eq:hessian-weighted-reciprocal-bound}, together with
\eqref{eq:hessian-full-maturity-sum}, imply, uniformly in $i$,
\begin{align}
    \left|
        \sum_{j\in\Pi_i}\delta_j
        \bigl(\bar X_{ij}^2-E[\bar X_{ij}^2]\bigr)
        \nabla_\theta^2\bigl(q_{ij}^{-1}\bigr)(\theta_0)
    \right|_2
    &\le
    C\sum_{j\in\Pi_i}\delta_jq_{ij}(\theta_0)
    \left|
        \nabla_\theta^2\bigl(q_{ij}^{-1}\bigr)(\theta_0)
    \right|\\
    &\le
    C\sum_{j\in\Pi_i}
    \delta_j(T_j-t_{i-1})^{\alpha-1}
    \le C.
\end{align}
Independence of the centered row sums therefore gives
\begin{equation}
\label{eq:gaussian-hessian-centered}
    E\left|
        \nabla_\theta^2\bar U_{n,d}(\theta_0)
        -\mathcal B_{n,d}(\theta_0)
    \right|^2
    \le\frac Cn.
\end{equation}
By Lemma~\ref{lem:deterministic-matrix-limits},
$\mathcal B_{n,d}(\theta_0)\to B(\theta_0)=B$.
Consequently,
\begin{equation}
\label{eq:gaussian-hessian-pointwise}
    \nabla_\theta^2\bar U_{n,d}(\theta_0)
    \xrightarrow{p}B.
\end{equation}

It remains to control the variation along the random
segment. Fix $r\in(0,1)$.
By Assumptions~\ref{ass:K0} and \ref{ass:K3} and compactness,
for every $\varepsilon_k>0$, there exists $\rho>0$ such that
\begin{equation}
\label{eq:truncated-hessian-parameter-continuity}
    \sup_{x\in[r,1]}
    \left[
        |k_\theta(x)-k_{\theta_0}(x)|
        +|\nabla_\theta k_\theta(x)
          -\nabla_\theta k_{\theta_0}(x)|
        +|\nabla_\theta^2k_\theta(x)
          -\nabla_\theta^2k_{\theta_0}(x)|
    \right]
    \le\varepsilon_k
\end{equation}
whenever $\theta\in\Theta$ and $|\theta-\theta_0|<\rho$.
Using the integral representations in Lemma~\ref{lemA2}
and the bounds $c\le q_{ij}(\theta)/\Delta_n\le C_r$
on the truncated region, we obtain from
\eqref{eq:truncated-hessian-parameter-continuity} that
\begin{align}
    &|\nabla_\theta^2\log q_{ij}(\theta)
      -\nabla_\theta^2\log q_{ij}(\theta_0)|+
    q_{ij}(\theta_0)
    |\nabla_\theta^2(q_{ij}^{-1})(\theta)
      -\nabla_\theta^2(q_{ij}^{-1})(\theta_0)|\\
    &\quad+
    |\nabla_\theta^2q_{ij}(\theta)
      -\nabla_\theta^2q_{ij}(\theta_0)|
    \le C_r\varepsilon_k
\end{align}
for all $i$, $j\in\Pi_i^r$, and
$\theta\in\Theta$ with $|\theta-\theta_0|<\rho$.
The constant $C_r$ is independent of $\varepsilon_k$ and
$\rho$.
Applying this bound to
\eqref{eq:contrast-hessian-formula} and taking expectations
gives
\begin{align}
\label{eq:truncated-hessian-local-modulus}
    E\left[
        \sup_{\substack{\theta\in\Theta\\
                        |\theta-\theta_0|<\rho}}
        |\nabla_\theta^2\bar U_{n,d}^r(\theta)
          -\nabla_\theta^2\bar U_{n,d}^r(\theta_0)|
    \right]
    &\le
    \frac{C_r\varepsilon_k}{n}
    \sum_{i=1}^n\sum_{j\in\Pi_i^r}\delta_j
    \left[
        1+\frac{E[\bar X_{ij}^2]}{q_{ij}(\theta_0)}
        +\frac14
    \right]\\
    &\le C_r\varepsilon_k.
\end{align}
The last inequality follows from the Gaussian moment
formula, Lemma~\ref{lemA1}, and
$\sum_i\sum_{j\in\Pi_i^r}\delta_j/n\le1$.

By the triangle inequality and Lemma~\ref{lemC1},
\begin{align}
\label{eq:gaussian-hessian-local-modulus}
    E\left[
        \sup_{\substack{\theta\in\Theta\\
                        |\theta-\theta_0|<\rho}}
        |\nabla_\theta^2\bar U_{n,d}(\theta)
          -\nabla_\theta^2\bar U_{n,d}(\theta_0)|
    \right]
    &\le
    2E\left[
        \sup_{\theta\in\Theta}
        |\nabla_\theta^2\bar U_{n,d}(\theta)
          -\nabla_\theta^2\bar U_{n,d}^r(\theta)|
    \right]
    +C_r\varepsilon_k\\
    &\le
    C(r+\Delta_d)^\alpha+C_r\varepsilon_k.
\end{align}
The definition of $\bar B_{n,d}$, convexity of $\Theta$,
and Markov's inequality now yield, for every $\varepsilon>0$,
\begin{equation}
    P\left(
        |\bar B_{n,d}
          -\nabla_\theta^2\bar U_{n,d}(\theta_0)|
        >\varepsilon
    \right)
    \le
    P(|\hat\theta_{n,d}-\theta_0|\ge\rho)
    +
    \frac{C(r+\Delta_d)^\alpha+C_r\varepsilon_k}{\varepsilon}.
\end{equation}
For fixed $r$ and $\varepsilon_k$, consistency makes the first
term tend to zero.
Letting $n,d\to\infty$, then $\varepsilon_k\downarrow0$,
and finally $r\downarrow0$, we obtain
\begin{equation}
    \bar B_{n,d}
    -\nabla_\theta^2\bar U_{n,d}(\theta_0)
    \xrightarrow{p}0.
\end{equation}
Combining this convergence with
\eqref{eq:integrated-hessian-transfer} and
\eqref{eq:gaussian-hessian-pointwise} in
\eqref{eq:integrated-hessian-decomposition}
proves the assertion.
\end{proof}

\begin{lem}
\label{lem:gaussian-proxy-score}
The random vectors
$\sqrt n\,\nabla_\theta\bar U_{n,d}(\theta_0)$
converge $\mathscr F_1$-stably in law to a centered
Gaussian random vector with covariance matrix $D$.
The limiting vector is defined on an extension of the
original probability space and is independent of
$\mathscr F_1$.
\end{lem}
\begin{proof}
For $i=1,\ldots,n$ and $j\in\Pi_i$, define
\begin{equation}
    Z_{ij}(t)
    \coloneqq
    \int_{t_{i-1}}^t
    \bar k_{\theta_0}^j(s)dW_s,
    \qquad t\in[t_{i-1},t_i],
    \qquad
    Z_{ij}\coloneqq Z_{ij}(t_i).
\end{equation}
Throughout this proof, we suppress the argument $\theta_0$
in $q_{ij}$, $q_i^{j\ell}$, and
$\nabla_\theta\log q_{ij}$.
Set
\begin{equation}
    \chi_i
    \coloneqq
    \frac1{\sqrt n}\sum_{j\in\Pi_i}
    \delta_j\nabla_\theta\log q_{ij}
    \left[
        1-\frac{Z_{ij}^2}{q_{ij}}+Z_{ij}
    \right].
\end{equation}
Since $\bar X_{ij}=Z_{ij}-q_{ij}/2$,
\eqref{eq:contrast-gradient-formula}, with $X_{ij}$
replaced by $\bar X_{ij}$, gives
\begin{equation}
\label{eq:proxy-score-sum}
    \sqrt n\,\nabla_\theta\bar U_{n,d}(\theta_0)
    =
    \sum_{i=1}^n\chi_i.
\end{equation}

The vector $(Z_{ij})_{j\in\Pi_i}$ is centered Gaussian
and independent of $\mathscr F_{t_{i-1}}$, with
covariances $E[Z_{ij}Z_{i\ell}]=q_i^{j\ell}$.
Lemma~\ref{lemA2}, applied with exponent $1-\alpha$,
and \eqref{eq:hessian-full-maturity-sum} yield
\begin{equation}
\label{eq:proxy-score-maturity-bound}
    \sum_{j\in\Pi_i}
    \delta_j|\nabla_\theta\log q_{ij}|
    \le
    C\sum_{j\in\Pi_i}
    \delta_j(T_j-t_{i-1})^{\alpha-1}
    \le C.
\end{equation}
Moreover, Lemma~\ref{lemA1} gives
$q_{ij}\le C\Delta_n^{2\alpha}$.

Write
$E_i[\cdot]\coloneqq
E[\cdot\mid\mathscr F_{t_{i-1}}]$.
Since $E_i[Z_{ij}]=0$ and $E_i[Z_{ij}^2]=q_{ij}$,
we have $E_i[\chi_i]=0$.
The Gaussian moment formula also gives
\begin{equation}
    E_i\left[
        \left(1-\frac{Z_{ij}^2}{q_{ij}}+Z_{ij}\right)
        \left(1-\frac{Z_{i\ell}^2}{q_{i\ell}}+Z_{i\ell}\right)
    \right]
    =
    \frac{2(q_i^{j\ell})^2}{q_{ij}q_{i\ell}}
    +q_i^{j\ell}.
\end{equation}
Consequently, for every $t\in[0,1]$,
\begin{equation}
    \sum_{i=1}^{\lfloor nt\rfloor}
    E_i[\chi_i\chi_i^\top]
    =
    \mathcal D_{n,d}(t,\theta_0)
    \longrightarrow D(t,\theta_0)
\end{equation}
by Lemma~\ref{lem:deterministic-matrix-limits}.
The limiting covariance process is continuous by
Assumption~\ref{ass:K4}.

Since centered Gaussian third moments vanish,
\begin{equation}
    E_i[\chi_i(W_{t_i}-W_{t_{i-1}})]
    =
    \frac1{\sqrt n}\sum_{j\in\Pi_i}
    \delta_j\nabla_\theta\log q_{ij}
    \int_{t_{i-1}}^{t_i}
    \bar k_{\theta_0}^j(s)ds.
\end{equation}
Thus, by the Cauchy--Schwarz inequality,
\eqref{eq:proxy-score-maturity-bound}, and
$q_{ij}\le C\Delta_n^{2\alpha}$,
\begin{align}
    \sum_{i=1}^n
    \left|
        E_i[\chi_i(W_{t_i}-W_{t_{i-1}})]
    \right|
    &\le
    \frac1{\sqrt n}\sum_{i=1}^n\sum_{j\in\Pi_i}
    \delta_j|\nabla_\theta\log q_{ij}|
    \sqrt{\Delta_n q_{ij}}\\
    &\le
    C\sqrt n\,\Delta_n^{\alpha+1/2}
    =
    C\Delta_n^\alpha
    \longrightarrow0.
\end{align}

The Gaussian moment formula, together with
$q_{ij}\le C$ and
\eqref{eq:proxy-score-maturity-bound}, gives
\begin{equation}
    \sum_{i=1}^n E_i[|\chi_i|^4]
    \le
    \frac C{n^2}\sum_{i=1}^n
    \left(
        \sum_{j\in\Pi_i}
        \delta_j|\nabla_\theta\log q_{ij}|
    \right)^4
    \le\frac Cn
    \longrightarrow0.
\end{equation}

Finally, let $N$ be any bounded
$(\mathscr F_t)$-martingale orthogonal to $W$.
By It\^o's formula,
\begin{equation}
    \chi_i
    =
    \int_{t_{i-1}}^{t_i}
    \frac1{\sqrt n}\sum_{j\in\Pi_i}
    \delta_j\nabla_\theta\log q_{ij}\,
    \bar k_{\theta_0}^j(s)
    \left(
        1-\frac{2Z_{ij}(s)}{q_{ij}}
    \right)dW_s.
\end{equation}
The corresponding square-integrable martingale is
orthogonal to $N$, so
\begin{equation}
    E_i[\chi_i(N_{t_i}-N_{t_{i-1}})]=0
\end{equation}
for every $i$.

Theorem~3-2 of \cite{Jacod1997}, applied with reference
martingale $W$ and underlying $\sigma$-field
$\mathscr F_1$, therefore yields
$\mathscr F_1$-stable convergence of
$\sum_{i=1}^n\chi_i$ to a random vector $G$ defined on
an extension of the original probability space, whose
conditional distribution given $\mathscr F_1$ is
$N(0,D)$.
Since $D$ is deterministic, $G$ is independent of
$\mathscr F_1$.
The assertion follows from \eqref{eq:proxy-score-sum}.
\end{proof}

\begin{proof}[Proof of Theorem~\ref{thm:stable-asymptotic-normality}]
Assumption~\ref{ass:K4} ensures that $B$ and $D$
are finite.
For any $y\in\mathbb R^q$,
\[
    y^\top By
    =
    4\int_0^1
    (1-x)
    \left(
        y^\top\nabla_\theta\log k_{\theta_0}(x)
    \right)^2dx.
\]
If this expression is zero, then
$y^\top\nabla_\theta\log k_{\theta_0}(x)=0$
for almost every $x\in(0,1)$.
Assumption~\ref{ass:K5} implies that $y=0$,
so $B$ is positive definite.

Similarly, if $y^\top Dy=0$, then
\[
    r\longmapsto
    \int_0^r
    y^\top\nabla_\theta\log k_{\theta_0}(x)dx
\]
is zero for almost every $r\in(0,1)$.
This function is absolutely continuous and hence
identically zero on $[0,1]$.
Its derivative is therefore zero almost everywhere,
and Assumption~\ref{ass:K5} again gives $y=0$.
Thus $D$ is also positive definite.

By Lemma~\ref{lem:integrated-hessian},
$B_{n,d}$ converges in probability to the positive
definite matrix $B$.
In particular, the probability that $B_{n,d}$ is
singular tends to zero.
For the remainder of this proof, we replace
$B_{n,d}$ by $I_q$ on its singularity event,
retaining the same notation, where $I_q$ denotes the $q \times q$ identity matrix.
This modification preserves convergence in probability
to $B$, makes $B_{n,d}$ invertible everywhere, and
preserves~\eqref{eq:estimator-taylor-expansion}
on an event whose probability tends to one.
By continuity of matrix inversion,
$B_{n,d}^{-1}$ converges in probability to $B^{-1}$.

Lemmas~\ref{lem:score-comparison}
and~\ref{lem:gaussian-proxy-score}, together with
\eqref{eq:maximal-mesh-condition}, imply that
$\sqrt n\,\nabla_\theta U_{n,d}(\theta_0)$
converges $\mathscr F_1$-stably in law to a centered
Gaussian random vector with covariance matrix $D$,
independent of $\mathscr F_1$.

Moreover, \eqref{eq:estimator-taylor-expansion}
implies that the difference between
$\sqrt n(\hat\theta_{n,d}-\theta_0)$ and
\begin{equation}
        -B_{n,d}^{-1}
    \sqrt n\,\nabla_\theta U_{n,d}(\theta_0)
\end{equation}
converges to zero in probability.
Slutsky's lemma for stable convergence now shows that
$\sqrt n(\hat\theta_{n,d}-\theta_0)$
converges $\mathscr F_1$-stably in law to a centered
Gaussian random vector $Z$ with covariance matrix
$B^{-1}DB^{-1}=\Gamma$.
The limiting vector is defined on an extension of the
original probability space and is independent of
$\mathscr F_1$.
\end{proof} 
\bibliographystyle{plain}
\bibliography{mybib}

\end{document}